\documentclass[11pt]{article}

 \usepackage{amsmath,amssymb,amscd,amsthm,esint}

\usepackage{graphics,amsmath,amssymb,amsthm,mathrsfs}

\usepackage{amsfonts}
\usepackage{esint}
\usepackage{pifont}
\usepackage{bbding}
\usepackage{latexsym}
\usepackage{mathrsfs}
\usepackage{amsmath}
\usepackage{amsthm}
\usepackage{amssymb}
\usepackage{graphicx}
\usepackage{lineno}
\usepackage{amssymb}

\usepackage[left=1in, right=1in, bottom=1.4in]{geometry}

\newtheorem{theorem}{Theorem}[section]
\newtheorem{corollary}{Corollary}[section]

\newtheorem{lemma}{Lemma}[section]
\newtheorem{remark}{Remark}[section]
\numberwithin{equation}{section}

\date{} 
\begin{document}

\title{Convergence Rates and Interior Estimates \\ in Homogenization of  Higher Order Elliptic Systems }

\author{Weisheng Niu\thanks{Supported in part by the NSF of China (11301003)
and Anhui Province (1708085MA02).},
 Zhongwei Shen\thanks{Supported in part by NSF grant DMS-1600520.}, and Yao Xu}

\maketitle
\pagestyle{plain}
\begin{abstract}
This paper is concerned with the quantitative homogenization of $2m$-order elliptic systems with bounded measurable, rapidly oscillating periodic coefficients.  We establish the sharp $O(\varepsilon)$  convergence rate in $W^{m-1, p_0}$ with
$p_0=\frac{2d}{d-1}$ in a bounded Lipschitz domain in $\mathbb{R}^d$
as well as the uniform large-scale interior $C^{m-1, 1}$ estimate. With additional smoothness assumptions, the uniform interior $C^{m-1, 1}$, $W^{m,p}$ and $C^{m-1, \alpha}$ estimates are also obtained. As applications of the regularity
estimates, we establish asymptotic expansions for fundamental solutions.

\end{abstract}



\section{Introduction}

 Let $\Omega$ be a bounded Lipschitz domain in $\mathbb{R}^d$. Consider the Dirichlet problem
 for a family of $2m$-order elliptic systems 
\begin{equation} \label{eq1}
 \begin{cases}
 \mathcal{L}_\varepsilon u_\varepsilon =f  &\text{ in } \Omega,  \vspace{0.3cm}\\
 Tr (D^\gamma u_\varepsilon)=g_\gamma  & \text{ on } \partial\Omega \quad \text{ for  } 0\leq|\gamma|\leq m-1,
\end{cases}
\end{equation}
where
$$ (\mathcal{L}_\varepsilon u_\varepsilon )_i  = (-1)^{m}\sum_{|\alpha|=|\beta|=m} D^\alpha \Big(A_{ij}^{\alpha \beta}\Big(\frac{x}{\varepsilon}\Big)D^\beta u_{\varepsilon j}\Big), ~~   1\leq i, j\leq  n,
$$
$ u_{\varepsilon j}$ denotes the $j$-th component of the $\mathbb{R}^n$-valued function $u_\varepsilon$, $\alpha, \beta,\gamma$
are multi-indices with nonnegative integer components $\alpha_k, \beta_k,\gamma_k, k=1,2,...,d$, and
$$ |\alpha|=\sum_{k=1}^d \alpha_k, ~~D^\alpha=D_{x_1}^{\alpha_1} D_{x_2}^{\alpha_2}\cdot\cdot\cdot D_{x_d}^{\alpha_d}.
$$
We assume that the coefficients matrix $A(y)=(A_{ij}^{\alpha \beta}(y))$ is real, bounded measurable with
\begin{align}\label{cod1}
\|A_{ij}^{\alpha \beta}(y)\|_{L^\infty(\mathbb{R}^d)}\leq \frac{1}{\mu},
\end{align}
and satisfies the coercivity condition
\begin{align}\label{cod2}
\sum_{|\alpha|=|\beta|=m} \int_{\mathbb{R}^d}
 D^\alpha \phi_i  A_{ij}^{\alpha \beta} D^\beta \phi_j \geq \mu \sum_{|\alpha|=m}\| D^\alpha \phi \|^2_{L^2(\mathbb{R}^d)}
\quad
   \text{ for any } \phi \in C_c^\infty(\mathbb{R}^d; \mathbb{R}^n ),
\end{align}
where  $\mu>0$.
We also assume that
\begin{align}\label{cod3}
A_{ij}^{\alpha \beta}(y+z)=A_{ij}^{\alpha \beta}(y) \ \  \text{ for any } y\in \mathbb{R}^d \text{ and } z\in \mathbb{Z}^d.
\end{align}
Functions satisfying condition (\ref{cod3}) will be called $1$-periodic. By a linear translation,  $ \mathbb{Z}^d$ in (\ref{cod3}) may be replaced by any lattice in $ \mathbb{R}^d.$

Let $W\!A^{m,p}(\partial\Omega, \mathbb{R}^n)$ denote the Whitney-Sobolev space  of
$\dot{g}=  \{g_\gamma\}_{|\gamma|\leq m-1}$, which is
the completion of the set of arrays of $\mathbb{R}^n$-valued functions
$$\left\{\{ D^\alpha \mathcal{G}\mid_{\partial\Omega}\}_{|\alpha|\leq m-1}:  \mathcal{G}\in C_c^\infty(\mathbb{R}^d; \mathbb{R}^n ) \right\},
$$
with respect to the norm
$$ \| \dot{g} \|_{W\!A^{m,p}(\partial\Omega)} =\sum_{|\alpha|\leq m-1} \|g_\alpha\|_{L^p(\partial \Omega)} + \sum_{|\alpha|=m-1} \|\nabla_{tan}  g_\alpha\|_{L^p(\partial \Omega)}.
$$
 Denote $W^{m,p}(\Omega; \mathbb{R}^n)$ the conventional Sobolev spaces of $\mathbb{R}^n$-valued  functions, and let $ W_0^{m,p}(\Omega; \mathbb{R}^n)$ be the completion of $C_c^\infty(\Omega;\mathbb{R}^n )$ in $W^{m,p}(\Omega; \mathbb{R}^n )$, with dual $W^{-m,p'}(\Omega; \mathbb{R}^n )$. Also following the conventions, we denote these spaces as $H^m(\Omega; \mathbb{R}^n ), H_0^m(\Omega; \mathbb{R}^n )$ and $H^{-m}(\Omega; \mathbb{R}^n )$
respectively when $p=2$.

It is well known that  under the ellipticity condition
(\ref{cod1})-(\ref{cod2}),
for any $\dot{g}\in W\!A^{m,2}(\partial\Omega, \mathbb{R}^n)$ and $ f\in H^{-m}(\Omega;\mathbb{R}^n)$,
the Dirichlet problem (\ref{eq1}) admits a unique weak solution $u_\varepsilon$ in $ H^m(\Omega;\mathbb{R}^n)$ such that
$$
 \int_\Omega \sum_{|\alpha|=|\beta|=m} D^\alpha v_i A_{ij}^{\alpha \beta}(x/\varepsilon)D^\beta u_j \, dx= \langle f, v\rangle
  ~~\text{ for any } v\in H_0^m(\Omega;\mathbb{R}^n ).
$$
Moreover,
 $$
 \|u_\varepsilon\|_{H^m(\Omega)} \leq C \left\{\|f\|_{H^{-m}(\Omega)} +\|\dot{g}\|_{W\!A^{m,2}(\partial\Omega)}\right\},
 $$
 where $C$ depends only on $d$, $m$, $n$, $\mu$ and $\Omega$.
Under the additional periodicity condition (\ref{cod3}),
the operator $\mathcal{L}_\varepsilon$ is G-convergent to  $\mathcal{L}_0 $, where
 \begin{align*}
(\mathcal{L}_0 u )_i = \sum_{|\alpha|=|\beta|=m} (-1)^m D^\alpha (\bar{A}_{ij}^{\alpha \beta}D^\beta u_j)
\end{align*}
is an elliptic operator of order $2m$ with constant coefficients,
 $$\bar{A}_{ij}^{\alpha \beta}=\sum_{|\gamma|=m}\frac{1}{|Q|}\int_Q \Big[A_{ij}^{\alpha \beta}(y)
 + A_{i\ell}^{\alpha \gamma}(y)D^\gamma \chi_{\ell j}^\beta (y)\Big]dy
 $$
 (see e.g.  \cite{zk}).
Here $Q=[-1/2, 1/2]^d$ and
$\chi= (\chi^\gamma_{ij})$ is the matrix of correctors for the operator $\mathcal{L}_\varepsilon$
(see Section 2 for definition). Furthermore,
  the matrix $(\bar{A}_{ij}^{\alpha \beta})$ is bounded and
satisfies the coercivity condition (\ref{cod2}). Thus the homogenized problem for (\ref{eq1}), given by
\begin{equation} \label{hoeq1}
 \begin{cases}
 \mathcal{L}_0  u_0 =f &\text{ in } \Omega,  \vspace{0.3cm}\\
 Tr (D^\gamma u_0)=g_\gamma & \text{ on } \partial\Omega \quad \text{ for }  0\leq|\gamma|\leq m-1,
\end{cases}
\end{equation}
admits a unique weak solution $u_0\in H^m(\Omega; \mathbb{R}^n )$,  satisfying
$$ \|u_0\|_{H^m(\Omega)} \leq C \left\{\|f\|_{H^{-m}(\Omega)} +\|\dot{g}\|_{W\!A^{m,2}(\partial\Omega)}\right\}.$$

Our first result gives the optimal convergence rate of $u_\varepsilon$ to $u_0$ in $W^{m-1,  2d/ (d-1)}(\Omega).$

\begin{theorem}\label{tcon}
Let $\Omega$ be a bounded Lipschitz domain in $\mathbb{R}^d$, $d\ge 2$.
Assume that the matrix $A=(A_{ij}^{\alpha\beta})$ satisfies (\ref{cod1})--(\ref{cod3}),
and is symmetric, i.e. $A=A^*$. Let $u_\varepsilon, u_0$ be the weak solutions
to the Dirichlet problems (\ref{eq1}) and  (\ref{hoeq1}), respectively.
Assume that $u_0\in H^{m+1}(\Omega;\mathbb{R}^n)$. Then
\begin{align}\label{convre1}
\|u_\varepsilon-u_0\|_{W_0^{m-1, q_0 } (\Omega)} \leq C \varepsilon \|u_0\|_{H^{m+1}(\Omega)},
\end{align}
where $q_0=\frac{2d}{d-1}$ and  $C$ depends only on $d, n, m, \mu$ and $\Omega.$
\end{theorem}

As a consequence of (\ref{convre1}), we obtain
  \begin{equation}\label{Higher-q-1-rate}
\|u_\varepsilon-u_0\|_{L^{q_1} (\Omega)} \leq C \varepsilon \|u_0\|_{H^{m+1}(\Omega)},
\end{equation}
  where $q_1=\frac{2d}{d-2m+1}$ if $d>2m-1$, $q_1\in (2, \infty)$ if $d=2m-1$, and
  $q=\infty$ if $d<2m-1$.

The problem of convergence rates, which is of great interest  in the quantitative homogenization theory,
 has been studied extensively for second-order elliptic equations. In particular,
 the optimal convergence rate in $L^2$,
 \begin{equation}\label{L-2-rate}
  \|u_\varepsilon-u_0\|_{L^{2} (\Omega)} \leq C \varepsilon \|u_0\|_{H^{2}(\Omega)},
  \end{equation}
 was established for second-order elliptic systems in divergence form.
 We refer readers to \cite{su,su1} and their references
 for general elliptic systems in $C^{1,1}$ domains and to
 \cite{klsa,klsj,sh1} for results in Lipshcitz domains.
 Also see related results in \cite{ps, gs1, klsc, sz}.
Moreover, the estimate
\begin{equation}\label{L-q-rate}
\|u_\varepsilon-u_0\|_{L^{q_0} (\Omega)}\leq C \varepsilon \|u_0\|_{H^{2}(\Omega)}
\end{equation}
 with $ q_0=2d/(d-1)$
 was proved for second-order elliptic systems with either Dirichlet or Neumann boundary conditions
 in Lipschitz domains in \cite{sh1}.

 Until very recently,
 few quantitative results were known for higher-order elliptic equations,
 although  qualitative convergence results (such as the G-convergence)
 have been obtained for many years \cite{jko,zk}.
 In \cite{pa,pa1,ks}  interesting results were established on the optimal convergence rates
 for higher-order elliptic equations in the whole space.
 In \cite{Suslina2017-D, Suslina2017-N}
 some two-parameter resolvent estimates were obtained
for  a general higher-order elliptic systems with periodic coefficients in a bounded $C^{2m}$ domain
$\Omega$ with homogeneous Dirichlet or Neumann  data.
In particular, it is proved that
 \begin{equation}\label{higher-L-2-rate}
 \| u_\varepsilon -u_0\|_{L^2(\Omega)} \le C \varepsilon \| f\|_{L^2(\Omega)}.
 \end{equation}
  Our Theorem \ref{tcon},
  which extends the estimate (\ref{L-q-rate}) for second-order elliptic systems,
  recovers the estimate (\ref{higher-L-2-rate}) if $\Omega$ is sufficiently smooth.

  Our next two results establish the uniform interior $C^{m-1, 1}$ and
  $W^{m, p}$ regularity of $u_\varepsilon.$

\begin{theorem}\label{tlip}
Assume that $ A(y)$ satisfies  (\ref{cod1})--(\ref{cod3}).
Let $u_\varepsilon\in H^m(B_R; \mathbb{R}^n)$ be a
weak solution to $\mathcal{L}_\varepsilon u_\varepsilon=\sum_{|\alpha|\le m-1} D^\alpha f^\alpha $ in a ball $B_R=B(x_0,R)$
with $f^\alpha \in L^q(B_R; \mathbb{R}^n)$ for some $ q>d$.
Then for $0<\varepsilon\leq r\le R/2<\infty,$ it holds that
\begin{align}\label{larlip}
\left(\fint_{B_r} |\nabla^m u_\varepsilon|^2\right)^{1/2}\leq C \left\{ \frac{1}{R^m} \left(\fint_{B_R}
|  u_\varepsilon|^2\right)^{1/2} + \sum_{|\alpha|\le m-1}
R^{m-|\alpha|}\left(\fint_{B_R} |f^\alpha|^q\right)^{1/q}\right\},
\end{align}
where $C$ depends only on $d,n,m, \mu$ and $q.$ If in addition,  $A$
is H\"{o}lder continuous, i.e.,
\begin{align}\label{hol}
|A(x)-A(y)|\leq \Lambda_0 |x-y|^{\tau_0}
\quad  \text{ for any }   x, y\in \mathbb{R}^d,
\end{align}
where   $  \Lambda_0>0$ and $  \tau_0\in(0,1)$, then
\begin{align}\label{lip}
  |\nabla^m u_\varepsilon(x_0)| \leq C \left\{ \frac{1}{R^m} \left(\fint_{B_R}
   |  u_\varepsilon|^2\right)^{1/2} + \sum_{|\alpha|\le m-1}
   R^{m-|\alpha|}
   \left(\fint_{B_R} |f^\alpha |^q\right)^{1/q}\right\},
\end{align}
for any $\varepsilon>0$,
where $C$ depends only on $d,n,m, \mu, q, \Lambda_0$ and $\tau_0.$
\end{theorem}

\begin{theorem}\label{twmp}
Suppose that $ A(y)$ satisfies  (\ref{cod1})--(\ref{cod3}) and
$A\in V\!M\!O (\mathbb{R}^d),$ i.e.,
\begin{align}\label{vmo}
  \sup_{x\in \mathbb{R}^d,  0<r<t} \fint_{B(x,r)}\Big|A(y)-\fint_{B(x,r)}A\Big|\, dy\leq \varrho(t),
\end{align}
for some nondecreasing continuous function $\varrho(t)$ on $[0,1]$  with $ \varrho(0)= 0$.
For $B=B(x_0,r)$ in $ \mathbb{R}^d$, let $u_\varepsilon\in H^m(2B; \mathbb{R}^n)$ be a weak solution to
$$ \mathcal{L}_\varepsilon u_\varepsilon=\sum_{|\alpha|\le m} D^\alpha f^\alpha  \text{ in } 2B~(=B(x_0, 2r)) $$
with $f^\alpha\in L^p(2B; \mathbb{R}^n)$ for some $2<p<\infty $.  Then
\begin{align}\label{wmp}
\left(\fint_{B}|\nabla^m u_\varepsilon|^p\right)^{1/p}\leq C
 \left\{ \frac{1}{r^m} \left(\fint_{2B}|u_\varepsilon|^2\right)^{1/2}
 +  \sum_{|\alpha|\le m} r^{m-|\alpha|}
 \left(\fint_{2B}|f^\alpha|^p\right)^{1/p}   \right\},
\end{align}
where $C$ depends only on $d, n, m, \mu, p$ and $\varrho(t)$ in (\ref{vmo}).
\end{theorem}

The regularity estimates  that are uniform in $\varepsilon>0$ are a central issue in quantitative homogenization.
For second-order elliptic systems
the study was initiated by M. Avellaneda and F. Lin in a series of celebrated papers
 \cite{al87, alo, al89, al91}. Using a compactness method,
  interior and boundary H\"older estimates, $W^{1,p}$ estimates and Lipschitz
  estimates were obtained for second-order elliptic systems
  with H\"{o}lder continuous coefficients and Dirichlet conditions in bounded $C^{1,\alpha}$ domains.
  The corresponding boundary estimates for solutions with Neumann conditions were obtained
  by C. Kenig, F. Lin and Z. Shen in \cite{klsa1}, using the compactness method.
  More recently, another scheme for uniform regularity estimates, especially in the large scale,
  was formulated  in \cite{ak}  and used for convex functionals with random coefficients.
  This scheme, which is based on convergence rates,
   was  further developed in \cite{as,sh1} for periodic and almost periodic second-order elliptic systems.
We refer the reader to \cite{gs, gs2, gsh, ks1,ks2} for
related results on uniform regularity estimates.

 Our Theorems \ref{tlip} and \ref{twmp} extend the interior uniform estimates for second-order elliptic systems
 to higher-order elliptic systems.
   As far as we know,  no uniform regularity result in the quantitative homogenization of higher-order elliptic equations
is previously known.

Let $ \Gamma^{\varepsilon, A}(x,y)$ denote the matrix of
   fundamental solutions associated to the operator $\mathcal{L}_\varepsilon\, (\varepsilon\geq 0)$.
   As applications of the regularity results above, the asymptotic behavior of $ \Gamma^{\varepsilon, A}(x,y)$
    is  derived.

\begin{theorem}\label{texpan}
  Assume that $A(y)$ satisfies  (\ref{cod1})--(\ref{cod3}) and (\ref{vmo}).
  Suppose that $2\le 2m<d$.
   Then for any multi-index $\zeta$ with $0\leq |\zeta|\leq m-1,$
  \begin{align}\label{expan1}
  |D_y^\zeta\Gamma^{\varepsilon,A}(x,y)-D_y^\zeta\Gamma^{0,A}(x,y)|\leq C \varepsilon |x-y|^{2m-d-|\zeta|-1},
   \end{align}
 for any $ x,y \in \mathbb{R}^d$ and $ x\neq y,$ where $C$ depends only on $d, n, m, \mu$ and the function $\varrho (t)$.
If in addition $A$ satisfies (\ref{hol}),
then for any multi-indices $\zeta,\xi,\eta$ with $0<  |\zeta|\leq m-1, |\xi|=|\eta|=m,$ we have
  \begin{equation} \label{expan2}
  \aligned
  \Big|D^\xi_x\Gamma_{ij}^{\varepsilon,A}(x,y)- & D^\xi_x\Gamma_{ij}^{0,A}(x,y)-   \sum_{|\gamma|=m}
  (D_x^\xi\chi_{ik}^\gamma) (x/\varepsilon) D_x^\gamma \Gamma_{kj}^{0,A}(x,y)  \Big|\\
&  \leq  \frac{C \varepsilon\ln (\varepsilon^{-1}|x-y|+1)} { |x-y|^{d+1-m} },
\endaligned
 \end{equation}
 \begin{equation}\label{expan3}
 \aligned
  \Big|D_x^\zeta D^\xi_y\Gamma_{ij}^{\varepsilon,A}(x,y)- & D_x^\zeta D^\xi_y\Gamma_{ij}^{0,A}(x,y)- \sum_{|\gamma|=m}
  D_x^\zeta D_y^\gamma \Gamma_{ik}^{0,A}(x,y)(D_y^\xi\chi_{jk}^{*\gamma}) (y/\varepsilon)\Big |\\
  & \leq
  \frac{C \varepsilon\ln (\varepsilon^{-1}|x-y|+1)} { |x-y|^{d+1+|\zeta|-m} },
  \endaligned
  \end{equation}
  \begin{equation}\label{expan4}
  |D_x^\eta D_y^\xi \Gamma_{ij}^{\varepsilon,A}(x,y) -  \Theta^{\eta,\xi}_{ij}(x,y) |
  \leq  \frac{C \varepsilon\ln (\varepsilon^{-1}|x-y|+1)} { |x-y|^{d+1} },
   \end{equation}
for any $ x,y \in \mathbb{R}^d$ and $ x\neq y,$ where $C$ depends only on $d, n, m, \mu, \Lambda_0$ and $\tau_0$, and  $$ \Theta^{\eta,\xi}_{ij}(x,y)=\sum_{|\sigma|=m} \sum_{|\gamma|=m} D_x^\eta\left\{ \frac{\delta_{ik}}{\sigma!} x^{\sigma}-  \varepsilon^m\chi_{ik}^{\sigma}  (x/\varepsilon ) \right\}D_x^{\sigma} D^\gamma_y\Gamma_{kl}^{0,A}(x,y)D_y^\xi \left\{\frac{\delta_{jl}}{\gamma!}y^\gamma -  \varepsilon^m\chi_{jl}^{*\gamma} (y/\varepsilon)  \right\}.$$
\end{theorem}

Formal asymptotic expansions of the fundamental solutions
for second-order elliptic operators in $\mathbb{R}^d$ were obtained using the method of Bloch waves \cite{se}. Later on, using the compactness method, the asymptotic behavior of fundamental solutions was studied by M. Avellaneda and F. Lin in \cite{al91}, and
the results were used  to prove  $L^p$ estimates for singular integrals associated with
$\mathcal{L}_\varepsilon$. The approach and results were further refined   by C. Kenig, F. Lin, and Z. Shen in \cite{klsc}, where the asymptotic behavior of Green and Neumann functions on bounded domains was investigated.
Our Theorem \ref{texpan} extends the results in \cite{al91, klsc}
for fundamental solutions to higher-order systems.

We now describe the main ideas in the proof of the main results of this paper.
Our proof of Theorem \ref{tcon}  follows the general scheme in \cite{su, sh1,sz}.
We consider the function
$$
w_\varepsilon = u_\varepsilon-u_0-\varepsilon^m\sum_{|\gamma|=m}\chi^\gamma({x}/{\varepsilon}) S^2_\varepsilon(D^\gamma \tilde{u}_0 )\rho_\varepsilon
$$
(see (\ref{w})  for the details).
This function not only allows us to deal with rough coefficients,
 but also avoids the use of boundary correctors, which are rather complicated for higher-order systems.
The key step in the proof of Theorem \ref{tcon} is to show that
\begin{align} \label{1.15}\| w_\varepsilon\|_{H^m_0(\Omega)} \leq C \varepsilon^{1/2}
 \left\{\|\dot{g}\|_{W\!A^{m,2}(\partial\Omega)}+\|f\|_{W^{-m+1,p_0}(\Omega)}\right\},
  \end{align}
where $p_0=2d/(d+1)$.
To this end we introduce the dual correctors for $\mathcal{L}_\varepsilon$ and use
the non-tangential maximal function estimates for the homogenized problem,
given in \cite{pv}. With (\ref{1.15}) at our disposal,
the estimate (\ref{convre1}) is obtained by a duality argument, motivated by \cite{su}.

Using the convergence result, we are able to drive the large-scale $C^{m-1,1}$ estimate (\ref{larlip}),
 following the ideas in \cite{as, ak}.
However, instead of estimating how well a solution $u_\varepsilon$ is approximated by  "affine" functions,
 we need to estimate how well $u_\varepsilon$ is approximated by  polynomials of  degree $m$.
  The full scale estimate (\ref{lip}) follows from (\ref{larlip}) through a standard blowup argument.

The proof of $W^{m,p}$ estimate (\ref{wmp}) uses the large-scale  estimate (\ref{larlip}) and is based on a
real variable arguments originated from \cite{ci} and further developed in \cite{szf,szt}. Using this approach, the $W^{m,p}$ estimates for $ \mathcal{L}_\varepsilon u_\varepsilon=\sum_{|\alpha|\le m} D^\alpha f^\alpha$  are reduced to a weak reverse H\"{o}lder inequality for  solutions to $ \mathcal{L}_\varepsilon u_\varepsilon =0$.

 Finally,  for the proof of Theorem \ref{texpan},  we remark that fundamental solutions for higher-order elliptic systems with rough coefficients are constructed by A. Barton recently in \cite{ab}. Since no smoothness conditions on the coefficients are required, the regularities derived there are very limited. Here with smoothness and periodicity conditions on the coefficients,  uniform size estimates for the fundamental solutions follow from the $C^{m-1, \theta}$
 estimates. Based on the regularity results in
 Theorem \ref{tlip}, we are
  able to prove Theorem \ref{texpan},  following the ideas of \cite{klsc}.


\section{Preliminaries}

\subsection{Correctors and dual correctors}

Set $Q=[-\frac{1}{2}, \frac{1}{2}]^d$, and let $H^m_{per}(Q; \mathbb{R}^n )$ denote
the closure of the set of 1-periodic functions in $C^\infty(\mathbb{R}^d ; \mathbb{R}^n )$
 with respect to the $H^m(Q;\mathbb{R}^n)$ norm.
For $1\leq i, j\leq n$ and  multi-index $\gamma$ with $ |\gamma|=m,$ the matrix of correctors $\chi= (\chi^\gamma_{ij})$ for the operator $\mathcal{L}_\varepsilon$ is given by the cell problem in $Q$,
\begin{equation*}
\begin{cases}
 \sum_{|\alpha|=|\beta|=m}  D^\alpha \big[A_{ik}^{\alpha \beta}(y)D^\beta \chi_{kj}^\gamma (y)\big]
 =- \sum_{|\alpha|=m}  D^\alpha  \big[A_{ij}^{\alpha \gamma}(y)\big]
 ~~    \text{ in } Q,\\
 \chi_j^\gamma (y)\in  H^m_{per}(Q; \mathbb{R}^n ),\\
 \int_Q \chi_j^\gamma (y)\, dy=0,
 \end{cases}
\end{equation*}
of which the existence of a unique solution $ \chi_j^\gamma (y)=\left( \chi_{1j}^\gamma (y), \chi_{2j}^\gamma (y),...., \chi_{nj}^\gamma (y)\right)$ for each $\gamma$ and $j$ is ensured by the Lax-Milgram theorem
(see e.g.  \cite{pa}).
In the same way, we introduce the matrix of correctors $ \chi^*= (\chi^{*\gamma}_{ij})$ for the adjoint operators $\mathcal{L}^*_\varepsilon$ of $\mathcal{L}_\varepsilon$, where
$$
\mathcal{L}^*_\varepsilon=  (-1)^{m}\sum_{|\alpha|=|\beta|=m} D^\alpha\Big(A^{*\alpha\beta}
\Big(\frac{x}{\varepsilon}\Big)D^\beta  \Big), \text{ with } ~A^*=(A_{ij}^{*\alpha\beta})= ( A_{ji}^{\beta \alpha}),~1\leq i, j\leq  n.
$$
Set
\begin{align*}
\bar{A}_{ij}^{\alpha \beta}=\frac{1}{|Q|}\int_Q  \Big\{A_{ij}^{\alpha \beta}(y)
+ \sum_{|\gamma|=m}A_{i\ell}^{\alpha \gamma}(y)D^\gamma \chi_{\ell j}^\beta (y) \Big\}dy.
\end{align*}
It is known  (\cite[Lemma 3.2]{pa}) that the constant matrix $\bar{A}=(\bar{A}_{ij}^{\alpha \beta})$
satisfies the coercivity condition,
\begin{align*}
\sum_{|\alpha|=|\beta|=m}\int_{\mathbb{R}^d}
 D^\alpha \phi_i\bar{A}_{ij}^{\alpha \beta}D^\beta \phi_j
 \, dx\geq \mu \sum_{|\alpha|=m}\| D^\alpha \phi \|^2_{L^2(\mathbb{R}^d)}
 \quad
 \text{ for any } \phi \in C_c^\infty(\mathbb{R}^d;\mathbb{R}^n).
\end{align*}
  The operator \begin{align*}
(\mathcal{L}_0 u )_i  = (-1)^m \sum_{|\alpha|=|\beta|=m}  D^\alpha (\bar{A}_{ij}^{\alpha \beta}D^\beta u_j)
\end{align*}
is the homogenized operator for the family of elliptic operators $\mathcal{L}_\varepsilon$.

For $1\leq i,j\leq n$ and multi-indices $\alpha,\beta$ with $|\alpha|=|\beta|=m,$ set
\begin{align}\label{duc}
B_{ij}^{\alpha \beta}(y)=A_{ij}^{\alpha \beta}(y)+ \sum_{|\gamma|=m}A_{ik}^{\alpha \gamma}(y) (D^\gamma \chi_{kj}^\beta)(y)-\bar{A}_{ij}^{\alpha \beta}.
\end{align}
By the definitions of  $\chi^\gamma(y)$ and $\bar{A}_{ij}^{\alpha \beta}$
we see that  $B_{ij}^{\alpha \beta}\in L^2(Q) $ is 1-periodic with zero mean and that
\begin{align*}\sum_{|\alpha|=m}D^\alpha B_{ij}^{\alpha \beta}(y)=0
\quad \text{ for all } \beta \text{ with } |\beta|=m \text{ and }
1\le i, j \le n.
\end{align*}

\begin{lemma}\label{l2.1}
For any $1\leq i,j\leq n$ and multi-indices  $\alpha, \beta$ with
$|\alpha|=|\beta|=m,$  there exists a function $ \mathfrak{B}_{ij}^{\gamma\alpha\beta} $ such that
 \begin{align*}
 \mathfrak{B}_{ij}^{\gamma\alpha\beta}  & =- \mathfrak{B}_{ij}^{\alpha\gamma\beta} , ~~\sum_{|\gamma|=m}D^\gamma \mathfrak{B}_{ij}^{\gamma\alpha\beta} = B_{ij}^{\alpha \beta} ,
\\
& \|\mathfrak{B}_{ij}^{\gamma\alpha\beta} \|_{H^m(Q)}\leq  C  \|B_{ij}^{\alpha,\beta}\|_{L^2(Q)},
 \end{align*}
  where $C$ depends only on $d, n, m.$
\end{lemma}
\begin{proof}
This was proved in \cite{pa}, using Fourier series. Here we present a different proof.
 Fix $i,j$ and $\beta$.
 Since $B_{ij}^{\alpha\beta} $  is a 1-periodic function in $L^2(Q)$ with zero mean, by the Lax-Milgram theorem,
 there exists a $b_{ij}^{\alpha\beta} \in H_{per}^{2m}(Q)$  such that $ \fint_Q b_{ij}^{\alpha\beta} =0$,
 \begin{align*}
  \sum_{|\gamma|=m}   D^\gamma \Big(D^\gamma b_{ij}^{\alpha\beta}\Big)
  =B_{ij}^{\alpha\beta} \quad \text{ in  } Q \quad \text{ and } \quad
   \|b_{ij}^{\alpha\beta}\|_{H^{2m}(Q)}\leq C  \|B_{ij}^{\alpha\beta}\|_{L^2(Q)}.
 \end{align*}
 Set  $$\mathfrak{B}_{ij}^{\gamma\alpha\beta}(y)= D^\gamma b_{ij}^{\alpha\beta}-D^\alpha b_{ij}^{\gamma\beta} . $$
   It is obvious that $\mathfrak{B}_{ij}^{\gamma\alpha\beta}=-\mathfrak{B}_{ij}^{\alpha\gamma\beta}$
   and $ \mathfrak{B}_{ij}^{\gamma\alpha\beta} \in H_{per}^m(Q)$ with
   $$
   \|\mathfrak{B}_{ij}^{\gamma\alpha\beta}\|_{H^m(Q)}\leq C  \|B_{ij}^{\alpha\beta}\|_{L^2(Q)}.
   $$
 Note that $ \sum_{|\alpha|=m} D^\alpha b_{ij}^{\alpha\beta} $ is 1-periodic and
   $$
      \sum_{|\gamma|=m}   D^\gamma  \Big[D^\gamma  \sum_{|\alpha|=m} D^\alpha b_{ij}^{\alpha\beta} \Big]
      = \sum_{|\alpha|=m} D^\alpha B_{ij}^{\alpha\beta}=0.
      $$
      It follows by the Liouville property for the operator
      $\sum_{|\gamma|=m} D^{2\gamma}$ that
      $ \sum_{|\alpha|=m} D^\alpha b_{ij}^{\alpha\beta} $ is constant. As a result,
 \begin{align*}
 \sum_{|\gamma|=m}D^\gamma\mathfrak{B}_{ij}^{\gamma\alpha\beta} = \sum_{|\gamma|=m}   D^\gamma (D^\gamma b_{ij}^{\alpha\beta})- \sum_{|\gamma|=m}  D^\gamma (D^\alpha b_{ij}^{\gamma\beta}) =B_{ij}^{\alpha\beta}  .
 \end{align*}
This completes the proof.
\end{proof}

 The function  $\mathfrak{B}= (\mathfrak{B}_{ij}^{\gamma\alpha\beta})$ is called the matrix of dual correctors for operators $ \mathcal{L}_\varepsilon.$
 As in the second-order case (see e.g. \cite{su,klsa, sh1}),
 it plays an important  role in the study of sharp convergence rates.

 \subsection{An $\varepsilon$-smoothing operator}

We fix $\varphi\in C_c^\infty(B(0,1/2))$ such that $\varphi\ge 0$ and $\int_{\mathbb{R}^d} \varphi (x)dx=1$.
Set  $\varphi_\varepsilon (x)=\frac{1}{\varepsilon^d} \varphi(\frac{x}{\varepsilon})$ and define
\begin{align*}
S_\varepsilon(f)(x)=\int_{\mathbb{R}^d} \varphi_\varepsilon(x-y)f(y)\, dy.
\end{align*}

\begin{lemma}\label{l2.2}
Assume that $f\in L^p(\mathbb{R}^d)$ for some $1\leq p<\infty$ and $g\in L_{loc}^p(\mathbb{R}^d)$. Let $h\in L^\infty(\mathbb{R}^d)$ with support $\mathcal{O}$. Then
\begin{align*}
\|g^\varepsilon S_\varepsilon(f) h \|_{L^p(\mathbb{R}^d)}
\leq C \sup_{x\in \mathbb{R}^d}\left( \fint_{B(x,1)} |g|^p dy\right)^{1/p} \|f\|_{L^p(\mathcal{O}^\varepsilon )}\| h\|_\infty,
\end{align*}
where $g^\varepsilon (x)=g(\varepsilon^{-1}x)$,
$\mathcal{O}^\varepsilon=\{x\in \mathbb{R}^d: dist(x,\mathcal{O})\leq \varepsilon \}$.  If in addition  $g$ is 1-periodic, then
\begin{align*}
\|g^\varepsilon S_\varepsilon(f) h \|_{L^p(\mathbb{R}^d)}\leq C   \|g\|_{L^p(Q)} \|f\|_{L^p(\mathcal{O}^\varepsilon )} \| h\|_\infty.
\end{align*}
\end{lemma}
\begin{proof}
The case $h=1$ is known (see e.g. \cite{sh1}). The general case follows from the observation that
$S_\varepsilon (f) h =S_\varepsilon (f \chi_{\mathcal{O}^\varepsilon}) h$.
\end{proof}

Let $ \Omega_{ \varepsilon}=\{x\in\Omega: dist(x, \partial \Omega)< \varepsilon\}$ and
$   \widetilde{\Omega}_\varepsilon= \{x\in \mathbb{R}^d:  dist(x,\partial\Omega)<\varepsilon\}.$

\begin{lemma}\label{l2.3}
Assume that  $f\in H^s(\mathbb{R}^d)$. Then for any multi-index $\alpha$ with $|\alpha|=s$,
\begin{align}
\|S_\varepsilon(D^\alpha f)\|_{L^2(\Omega_{\varepsilon})}\leq C \varepsilon^{-s}\|f\|_{L^2( \widetilde{\Omega}_{2\varepsilon})},\label{l2.3re1}\\
\|S_\varepsilon(D^\alpha f   )\|_{L^2(\Omega \setminus \Omega_{2\varepsilon} )}\leq C \varepsilon^{-s} \|f\|_{L^2(\Omega\setminus\Omega_{\varepsilon})}  .\label{l2.3re2}
\end{align}

\end{lemma}
\begin{proof}
Using integration by parts and the Cauchy inequality, we see that
\begin{align*}
\|S_\varepsilon(D^\alpha f  )\|^2_{L^2(\Omega_\varepsilon)} &=\int_{\Omega_\varepsilon}  \Big|\int_{\mathbb{R}^d}
D^\alpha \varphi_\varepsilon (x-y)  f(y)     \, dy \Big|^2   \, dx\nonumber\\
&\leq  \frac{C}{\varepsilon^{s}}
\int_{\Omega_\varepsilon} \int_{\widetilde{\Omega}_{2\varepsilon}}
|D^\alpha\varphi_\varepsilon (x-y)|\, |f(y)|^2 \, dy dx\nonumber\\
&\leq \frac{C}{\varepsilon^{2s}}
\int_{\widetilde{\Omega}_{2\varepsilon}} |f(y) |^2  dy\nonumber,
\end{align*}
where we also used Fubini's Theorem for the last inequality.
This gives (\ref{l2.3re1}).
 The proof of (\ref{l2.3re2}) is similar.
\end{proof}

Let $\nabla^s f =\big( D^\alpha f \big)_{|\alpha|=s}$.

\begin{lemma}\label{l2.4}
Let $f\in W^{1,q}(\mathbb{R}^d)$ for some $1\leq q<\infty$. Then
\begin{align}\label{l2.4re1}
\|  S_\varepsilon(f)-f\|_{L^q(\mathbb{R}^d)}\leq C \varepsilon   \|\nabla f\|_{L^q(\mathbb{R}^d)},
\end{align}
Furthermore,  if  $f\in W^{s,p}(\mathbb{R}^d)$, where $p=\frac{2d}{d+2k-1}$ and $ 1\le k <\frac{d+1}{2} $,
  then
\begin{align}
&\|  S_\varepsilon(\nabla^s f)\|_{L^2(\mathbb{R}^d)}\leq C \varepsilon^{-s-k+\frac{1}{2}}     \|f\|_{L^p(\mathbb{R}^d)}
\quad \text{ if } \ \ s\ge 0, \label{l2.4re2}\\
&\| S_\varepsilon(f)\|_{L^2(\mathbb{R}^d)}\leq C  \varepsilon^{s-k+\frac{1}{2}}  \|\nabla^sf\|_{L^p(\mathbb{R}^d )} \quad
\text{ if } \ \ 0\leq s\le k-1, \label{l2.4re3}\\
&\|  S_\varepsilon(f)-f\|_{L^2(\mathbb{R}^d)}\leq C \varepsilon^{s-k+\frac{1}{2}}     \|\nabla^s f\|_{L^p(\mathbb{R}^d)}
\quad \text{ if } \ \ 0\leq s\le k.\label{l2.4re4}
\end{align}
\end{lemma}

\begin{proof}
The inequality (\ref{l2.4re1}) is known (see e.g.  \cite[Lemma 2.2]{sh1} for a proof).
By Parseval's theorem and H\"{o}lder's  inequality, we have
\begin{align}
&\int_{\mathbb{R}^d} |S_\varepsilon(\nabla^sf)|^2dx=
(2\pi)^{2s}\int_{\mathbb{R}^d} |\widehat{\varphi}(\varepsilon\xi)|^2 | \xi|^{2s}|\widehat{f}(\xi)|^2d\xi\nonumber\\
& \leq (2\pi)^{2s}\left\{\int_{\mathbb{R}^d}
\left( |\widehat{\varphi}(\varepsilon\xi)|| \xi|^{s} \right)^\frac{2d}{2k-1}   d\xi\right\}^\frac{2k-1}{d}
 \left\{\int_{\mathbb{R}^d} |\widehat{f}(\xi)|^\frac{2d}{d-2k+1}d\xi\right\}^\frac{d-2k+1}{d}   \nonumber\\
& =(2\pi)^{2s}  \varepsilon^{-2s-2k+1} \left\{\int_{\mathbb{R}^d} \left( |\widehat{\varphi}(\xi)|| \xi|^{s} \right)^\frac{2d}{2k-1}d\xi\right\}^\frac{2k-1}{d} \|\widehat{f}\|^2_{L^\frac{2d}{d-2k+1}(\mathbb{R}^d )}\nonumber\\
&\leq C  \varepsilon^{-2s-2k+1}  \|f\|^2_{L^\frac{2d}{d+2k-1}(\mathbb{R}^d )},\nonumber
\end{align}
where we have used  the Hausdorff-Young inequality in the last step. This gives (\ref{l2.4re2}).

To prove (\ref{l2.4re3}), we note that
\begin{align}
&\int_{\mathbb{R}^d} |S_\varepsilon(f)|^2dx=\int_{\mathbb{R}^d} |\widehat{\varphi}(\varepsilon\xi)|^2|\xi|^{-2s} |\xi|^{2s}|\widehat{f}(\xi)|^2d\xi\nonumber\\
& \leq \left\{\int_{\mathbb{R}^d} \left(|\xi|^{-s}|\widehat{\varphi}(\varepsilon\xi)|\right)^\frac{2d}{2k-1}d\xi\right\}^\frac{2k-1}{d}   \left\{\int_{\mathbb{R}^d} \left(|\xi|^{s}|\widehat{f}(\xi)|\right)^\frac{2d}{d-2k+1}d\xi\right\}^\frac{d-2k+1}{d}   \nonumber\\
& \leq C \varepsilon^{-2k+1+2s} \left\{\int_{\mathbb{R}^d} \left(|\xi|^{-s}|\widehat{\varphi}(\xi)|\right)^ \frac{2d}{2k-1}d\xi\right\}^\frac{2k-1}{d} \||\xi|^{s}\widehat{f}\|^2_{L^\frac{2d}{d-2k+1}(\mathbb{R}^d )}\nonumber\\
& \leq C \varepsilon^{-2k+1+2s} \left\{\int_{|\xi|\leq 1} \left( \frac{1} { |\xi|^{s}} \right)^ \frac{2d}{2k-1}d\xi+
 \int_{|\xi|> 1}   |\widehat{\varphi}(\xi)| ^ \frac{2d}{2k-1}d\xi
  \right\}^\frac{2k-1}{d} \|\widehat{\nabla^sf}\|^2_{L^\frac{2d}{d-2k+1}(\mathbb{R}^d )}\nonumber\\
&\leq C  \varepsilon^{-2k+1+2s}  \|\nabla^sf\|^2_{L^\frac{2d}{d+2k-1}(\mathbb{R}^d )},\nonumber
\end{align}
where in the last step we have used the Hausdorff-Young inequality and also the fact that
$ \frac{2ds}{2k-1}< d \ (\text{since } s\le k-1)$.

Finally, since $ \widehat{\varphi}(0)=\int_{\mathbb{R}^d}\varphi (x) dx=1,$
\begin{align*}
&\int_{\mathbb{R}^d} |S_\varepsilon(f)-f|^2dx=\int_{\mathbb{R}^d} |\widehat{\varphi}(\varepsilon\xi)-\widehat{\varphi}(0)|^2 |\widehat{f}|^2d\xi\nonumber\\
& \leq \left\{\int_{\mathbb{R}^d}(|\widehat{\varphi}(\varepsilon\xi)-\widehat{\varphi}(0)| |\xi|^{-s})^\frac{2d}{2k-1} d\xi\right\}^\frac{2k-1}{d}  \||\xi|^s\widehat{f}\|^2_{L^\frac{2d}{d-2k+1}(\mathbb{R}^d )}\nonumber\\
&\leq C \varepsilon^{2s-2k+1}  \left\{\int_{\mathbb{R}^d} (|\widehat{\varphi}(\xi)-\widehat{\varphi}(0)|  |\xi|^{-s})^\frac{2d}{2k-1} d\xi\right\}^\frac{2k-1}{d}    \|\widehat{\nabla^sf}\|^2_{L^\frac{2d}{d-2k+1}(\mathbb{R}^d )}\nonumber\\
&\leq C \varepsilon^{2s-2k+1}  \|\nabla^sf\|^2_{L^\frac{2d}{d+2k-1}(\mathbb{R}^d )},
\end{align*}
which is exactly (\ref{l2.4re4}). For the last inequality, we have used  $| \widehat{\varphi}(\xi)-\widehat{\varphi}(0)|\leq C |\xi| $, the
assumption  $s\le k$ and the Hausdorff-Young inequality.
\end{proof}

\begin{lemma}\label{l2.5}
Let $\Omega$ be a bounded Lipschitz domain in $\mathbb{R}^d$ and $f\in H^{m+1}( \mathbb{R}^d)$.
 Then
\begin{align*}
\int_{\widetilde{\Omega}_\varepsilon} |\nabla^m f(x)|^2dx\leq C \varepsilon \|f\|^2_{H^{m+1}( \mathbb{R}^d)}.
\end{align*}

\end{lemma}

\begin{proof}
See e.g.  \cite{ps} or  \cite{sh1} for the case $m=0$.
The case $m\ge 1$ follows by applying the estimate to the function $\nabla^mf$.
\end{proof}

The following is a Caccioppoli inequality  for higher-order elliptic systems.

\begin{lemma}\label{l2.6}
Assume that $A$ satisfies (\ref{cod1})--(\ref{cod2}). Let $u\in H^m(2B; \mathbb{R}^n )$ be a weak
solution to $\mathcal{L}_1  u  = \sum_{|\alpha|\le m}D^\alpha f^\alpha$ in $2B$ for a ball $B=B(x_0, r),
$ where $f^\alpha \in L^2(2B; \mathbb{R}^n )$. Then for $1\leq k\leq m$, we have
\begin{align}
\int_B |\nabla^k u  |^2\leq   \frac{C}{r^{2k}} \int_{2B}|u  |^2
+C\sum_{|\alpha|\le m} r^{4m-2|\alpha|-2k} \int_{2B} |f^\alpha|^2,\label{l2.6re2}
\end{align}
where $C$ depends only on $d$, $m$, $n$ and $\mu$.
\end{lemma}

\begin{proof}
See e.g. \cite{cs} and \cite{ab}.
\end{proof}


\section{Convergence rates}

Let $\rho_\varepsilon$ be a function in $C_c^\infty(\Omega)$ satisfying the following conditions,
\begin{align*}
 \begin{split}supp(\rho_\varepsilon)\subset\{x\in\Omega: dist(x, \partial \Omega)\geq  3\varepsilon\} = \Omega\setminus \Omega_{3\varepsilon},\\
 0\leq\rho_\varepsilon\leq 1,~\rho_\varepsilon=1  \text{ on } \Omega\setminus \Omega_{4\varepsilon},
 \text{ and } |\nabla^k \rho_\varepsilon|\le C \varepsilon^{-k} \text{ for } 1\le k\le m.
\end{split}\end{align*}

\begin{lemma}\label{l3.1}
Let $\Omega$ be a bounded Lipschitz domain in $\mathbb{R}^d$.
Assume that the matrix $ A$ satisfies (\ref{cod1})--(\ref{cod3}).
Let $u_\varepsilon~(\varepsilon\geq0)$ be the weak solution to Dirichlet problem (\ref{eq1}).
Suppose that $u_0\in H^{m+1}(\Omega;\mathbb{R}^d)$
with $\tilde{u}_0\in H^{m+1}(\mathbb{R}^d;\mathbb{R}^n)$ being its extension.
Define \begin{align}\label{w}
w_\varepsilon = u_\varepsilon-u_0-\varepsilon^m\sum_{|\gamma|=m}\chi^\gamma(\frac{x}{\varepsilon}) S^2_\varepsilon(D^\gamma \tilde{u}_0 )\rho_\varepsilon,\end{align}
  where $ S^2_\varepsilon= S_\varepsilon \circ S_\varepsilon $.
  Then for any $\phi \in H^m_0(\Omega;\mathbb{R}^n),$ we have
\begin{equation}
\label{l3.1re1}
\aligned
 & \Big|\sum_{|\alpha|=|\beta|=m} \int_\Omega D^\alpha \phi_i  A_{ij}^{\alpha\beta}(\frac{x}{\varepsilon})
 D^\beta w_{\varepsilon j}\Big|\\
& \leq C \| \nabla^m \phi\|_{L^2(\Omega_{4\varepsilon})}\| \nabla^m u_0\|_{L^{2}(\Omega_{4\varepsilon})}
 + C\| \nabla^m \phi\|_{L^2(\Omega)} \|  \nabla^m \tilde{u}_0
 -S_\varepsilon(\nabla^m \tilde{u}_0)  \|_{L^2(\Omega\setminus \Omega_{2\varepsilon})}\\
 &+ C  \| \nabla^m \phi\|_{L^2(\Omega_{4\varepsilon})} \sum_{ 0\leq k \leq m-1,  }\varepsilon^{k } \| S_\varepsilon(\nabla^{m+k}  \tilde{u}_0 )\|_{L^2(\Omega_{5\varepsilon}\setminus \Omega_{2\varepsilon})}\\
&+  C \|\nabla^m \phi\|_{L^2(\Omega)} \sum_{0\leq k\leq m-1}\varepsilon^{m-k } \|  S_\varepsilon(\nabla^{2m-k}  \tilde{u}_0    )\|_{L^2(\Omega\setminus \Omega_{2\varepsilon})},
\endaligned
\end{equation}
where $C$ depends only on $d, n, m, \mu$ and $\Omega.$
Consequently,
\begin{align}\label{l3.1re2}
\Big|\sum_{|\alpha|=|\beta|=m} \int_\Omega D^\alpha \phi_i A_{ij}^{\alpha\beta}(\frac{x}{\varepsilon}) & D^\beta w_{\varepsilon j}\Big |
 \leq  C \|u_0\|_{H^{m+1}(\Omega)}\left\{\varepsilon \|\nabla^m \phi\|_{L^2(\Omega)}  +\varepsilon^{1/2}  \|\nabla^m \phi\|_{L^2(\Omega_{4\varepsilon})} \right \}.
\end{align}
\end{lemma}

\begin{proof}
For the simplicity of presentation, we will omit the subscripts $i,j$.
Using the definitions of $w_\varepsilon$ and $B$ in (\ref{duc}),
a direct computation shows that  for any $\phi\in H^m_0(\Omega;\mathbb{R}^n)$,
 \begin{equation}\label{pl311}
 \aligned
 & \sum_{|\alpha|=|\beta|=m}  \int_\Omega D^\alpha \phi
A^{\alpha\beta}(\frac{x}{\varepsilon})  D^\beta w_\varepsilon\, dx   \\
 & =  - \sum_{|\alpha|=|\beta|=m}\int_\Omega   D^\alpha \phi
\Big\{ \Big[A^{\alpha \beta}(\frac{x}{\varepsilon})-\bar{A}^{\alpha \beta}\Big]
 \Big[D^\beta u_0-S^2_\varepsilon(D^\beta \tilde{u}_0)\rho_\varepsilon\Big] \Big\}   \\
&\ \ \ - \sum_{\substack{|\alpha|=|\beta|=|\gamma|=m\\ \zeta+\eta =\beta\\  0\leq|\zeta|\leq m-1}}
C(\zeta,\eta)\varepsilon^{m-|\zeta|}  \int_\Omega D^\alpha \phi \left\{A^{\alpha \beta}(\frac{x}{\varepsilon}) (D^{\zeta}\chi^\gamma)(\frac{x}{\varepsilon})  D^{\eta} \Big[ S^2_\varepsilon(D^\gamma \tilde{u}_0 )\rho_\varepsilon\Big]\right\} \\
&\ \ \ -\sum_{|\alpha|=|\beta|=m}\int_\Omega D^\alpha \phi
 B^{\alpha \beta}(\frac{x}{\varepsilon})  S^2_\varepsilon(D^\beta \tilde{u}_0)\rho_\varepsilon   \\
&\doteq~ I_1+ I_2+ I_3.
\endaligned
\end{equation}
Using
\begin{equation}
\aligned
& D^\beta u_0-S^2_\varepsilon(D^\beta \tilde{u}_0)\rho_\varepsilon\\
& =\big[ D^\beta u_0-(D^\beta \tilde{u}_0) \rho_\varepsilon\big]
+ \big[ D^\beta \tilde{u}_0- S_\varepsilon(D^\beta \tilde{u}_0)\big]\rho_\varepsilon
+\big[S_\varepsilon ( D^\beta \tilde{u}_0) - S^2_\varepsilon(D^\beta \tilde{u}_0 ) \big]\rho_\varepsilon,
\endaligned
\end{equation}
 we obtain
\begin{align}\label{pl312}
|I_1|
 \leq C    \| \nabla^m \phi\|_{L^2(\Omega_{4\varepsilon})}\| \nabla^m u_0\|_{L^{2}(\Omega_{4\varepsilon})}
 +C\| \nabla^m \phi\|_{L^2(\Omega)} \| \nabla^m \tilde{u}_0-S_\varepsilon(\nabla^m \tilde{u}_0) \|_{L^2(\Omega\setminus \Omega_{2\varepsilon})}.
\end{align}

To deal with $I_2,$ we observe that for  $|\alpha|=|\beta|=|\gamma|=m,$
\begin{align} \label{pl313}
&\sum_{\zeta+\eta =\beta, 0\leq|\zeta|\leq m-1} C(\zeta,\eta) \varepsilon^{m-|\zeta|} \int_\Omega D^\alpha \phi \left\{A^{\alpha \beta}(\frac{x}{\varepsilon}) (D^{\zeta}\chi^\gamma)(\frac{x}{\varepsilon})  D^{\eta} \left[S^2_\varepsilon(D^\gamma \tilde{u}_0) \rho_\varepsilon \right]\right\}  \nonumber\\
&=\varepsilon^{m}     \int_\Omega D^\alpha \phi \left\{A^{\alpha \beta}(\frac{x}{\varepsilon})
\chi^\gamma  (\frac{x}{\varepsilon})   S^2_\varepsilon(D^{\beta}D^\gamma \tilde{u}_0 )  \rho_\varepsilon \right \}  \nonumber\\
&+  \varepsilon^{m }   \sum_{ \eta^\prime+\eta^{\prime\prime}
=\beta,  |\eta''|\geq1 }C(\eta',\eta'') \int_\Omega D^\alpha \phi \left\{A^{\alpha \beta}(\frac{x}{\varepsilon}) \chi^\gamma (\frac{x}{\varepsilon})   S^2_\varepsilon(D^{\eta' }D^\gamma \tilde{u}_0) D^{\eta'' } \rho_\varepsilon \right\}   \nonumber\\
&+    \sum_{\zeta+ \eta=\beta, 1\leq |\zeta|,|\eta|   }\varepsilon^{m-|\zeta|}
C(\zeta,\eta) \int_\Omega D^\alpha \phi \left\{A^{\alpha \beta}(\frac{x}{\varepsilon}) (D^{\zeta}\chi^\gamma)(\frac{x}{\varepsilon})   S^2_\varepsilon(D^{\eta}D^\gamma \tilde{u}_0 )  \rho_\varepsilon\right\}  \nonumber\\
&+  \sum_{\zeta+\eta'+\eta''=\beta, 1\leq  |\zeta|, |\eta''|  } \varepsilon^{m-|\zeta|}
C(\zeta,\eta',\eta'')\int_\Omega D^\alpha \phi \left\{A^{\alpha \beta}(\frac{x}{\varepsilon}) (D^{\zeta}\chi^\gamma)(\frac{x}{\varepsilon})   S^2_\varepsilon(D^{\eta'}D^\gamma \tilde{u}_0 ) D^{\eta''} \rho_\varepsilon\right\}    \nonumber\\
& \doteq I_{21}+I_{22}+I_{23}+I_{24}.
\end{align}
Note that by Cauchy inequality and Lemma \ref{l2.2},
\begin{align}\label{pl314}
 |I_{21}|&\leq C \varepsilon^{m }  \| \nabla^m \phi\|_{L^2(\Omega)} \| \chi^\gamma  (\frac{x}{\varepsilon})   S^2_\varepsilon(\nabla^{2m} \tilde{u}_0   ) \rho_\varepsilon  \|_{L^2(\Omega)}\nonumber\\
 & \leq C \varepsilon^{m }  \| \nabla^m \phi\|_{L^2(\Omega)} \| S_\varepsilon( \nabla^{2m} \tilde{u}_0 )    \|_{L^2(\Omega\setminus\Omega_{2\varepsilon} )},
\end{align}
\begin{align}\label{pl315}
 |I_{23}|&\leq C  \sum_{1\leq k\leq m-1}\varepsilon^{m-k}  \| \nabla^m \phi\|_{L^2(\Omega)} \| \nabla^k \chi^\gamma  (\frac{x}{\varepsilon})   S^2_\varepsilon(\nabla^{2m-k}  \tilde{u}_0   )\rho_\varepsilon \|_{L^2(\Omega)}\nonumber\\
 & \leq C  \sum_{1\leq k\leq m-1}\varepsilon^{m-k }  \| \nabla^m \phi\|_{L^2(\Omega)} \|  S_\varepsilon(\nabla^{2m-k}  \tilde{u}_0   ) \|_{L^2(\Omega\setminus\Omega_{2\varepsilon} )}.
\end{align}
Similarly, we have
 \begin{align}\label{pl316}
 |I_{22}|&\leq C  \sum_{1\leq k\leq m}\varepsilon^{m}  \| \nabla^m \phi\|_{L^2(\Omega_{4\varepsilon})} \|  \chi^\gamma  (\frac{x}{\varepsilon})
 S^2_\varepsilon(\nabla^{2m-k}  \tilde{u}_0  )\nabla^{k} \rho_\varepsilon\|_{L^2(\Omega_{4\varepsilon})}\nonumber\\
 &\leq C \sum_{0\leq k\leq m-1}\varepsilon^{k }  \| \nabla^m \phi\|_{L^2(\Omega_{4\varepsilon})} \|  S_\varepsilon(\nabla^{m+k}  \tilde{u}_0) \|_{L^2(\Omega_{5\varepsilon}\setminus \Omega_{2\varepsilon})},
\end{align}
and
\begin{align}\label{pl317}
 |I_{24}|&\leq C  \sum_{k_1+k_2+k_3= m, k_1,k_3\geq 1}\varepsilon^{m-k_1}  \| \nabla^m \phi\|_{L^2(\Omega_{4\varepsilon})} \| \nabla^{k_1} \chi^\gamma  (\frac{x}{\varepsilon})   S^2_\varepsilon(\nabla^{m+k_2}  \tilde{u}_0 ) \nabla^{k_3} \rho_\varepsilon\|_{L^2(\Omega_{4\varepsilon})}\nonumber\\
 &\leq C  \sum_{0\leq k\leq m-2}\varepsilon^{k}  \| \nabla^m \phi\|_{L^2(\Omega_{4\varepsilon})} \| S_\varepsilon(\nabla^{m+k}  \tilde{u}_0 ) \|_{L^2(\Omega_{5\varepsilon}\setminus \Omega_{2\varepsilon})}.
\end{align}
By combining (\ref{pl313})--(\ref{pl317}), we obtain that
\begin{align}\label{pl318}
|I_2 | \leq & C  \sum_{0\leq k\leq m-1}\varepsilon^{m-k }  \| \nabla^m \phi\|_{L^2(\Omega)} \|  S_\varepsilon(\nabla^{2m-k}  \tilde{u}_0    )\|_{L^2(\Omega\setminus \Omega_{2\varepsilon})}\nonumber\\
&+ C   \sum_{ 0\leq k \leq m-1 }\varepsilon^{k } \| \nabla^m \phi\|_{L^2(\Omega_{4\varepsilon})} \| S_\varepsilon(\nabla^{m+k} \tilde{u}_0 )\|_{L^2(\Omega_{5\varepsilon}\setminus \Omega_{2\varepsilon})},
\end{align}
 where
$ C $    depends only on $d,n,m,\mu$ and $\Omega$.

Now let us turn to $I_3.$ Using Lemma \ref{l2.1}, we deduce that
\begin{align*}
 I_3&=-\varepsilon^m \sum_{|\alpha|=|\beta|=|\gamma|=m}\int_\Omega  D^\alpha \phi   (D^\gamma \mathfrak{B}^{\gamma \alpha \beta})(\frac{x}{\varepsilon})  S^2_\varepsilon(D^\beta \tilde{u}_0 ) \rho_\varepsilon     \nonumber \\
&=\varepsilon^m \sum_{|\alpha|=|\beta|=|\gamma|=m} (-1)^{m+1}\int_\Omega  D^\gamma D^\alpha \phi    \mathfrak{B}^{\gamma\alpha \beta}(\frac{x}{\varepsilon})    S^2_\varepsilon(D^\beta \tilde{u}_0) \rho_\varepsilon     ~~~~~~ \nonumber\\
&\quad
+\varepsilon^{m} \sum_{\substack{|\alpha|=|\beta|=|\gamma|=m\\
\zeta+\eta=\gamma, \ 0\leq|\zeta|\leq m-1}}
C (\zeta, \eta) \int_\Omega D^\zeta D^\alpha \phi   \mathfrak{B}^{\gamma\alpha \beta}(\frac{x}{\varepsilon})  D^\eta \left[ S^2_\varepsilon(D^\beta \tilde{u}_0) \rho_\varepsilon \right]     \nonumber\\
&= \sum_{\substack{ |\alpha|=|\beta|=|\gamma|=m\\ \zeta'+\eta' =\gamma,\  0\leq|\zeta'|\leq m-1}}
C (\zeta^\prime, \eta^\prime)
 \varepsilon^{m-|\zeta'| } \int_\Omega D^\alpha \phi
(D^{\zeta'}\mathfrak{B}^{\gamma\alpha \beta}) (\frac{x}{\varepsilon})D^{\eta'}
\left[ S^2_\varepsilon(D^\beta \tilde{u}_0) \rho_\varepsilon\right]  , \nonumber
\end{align*}
which may be handled in the same manner as $I_2$.
As a result,
\begin{align}\label{pl319}
|I_3|  \leq  & C  \sum_{0\leq k\leq m-1}\varepsilon^{m-k }  \| \nabla^m \phi\|_{L^2(\Omega)} \|  S_\varepsilon(\nabla^{2m-k}  \tilde{u}_0    )\|_{L^2(\Omega\setminus \Omega_{2\varepsilon})}\nonumber\\
&+ C   \sum_{ 0\leq k \leq m-1,  }\varepsilon^{k } \| \nabla^m \phi\|_{L^2(\Omega_{4\varepsilon})} \| S_\varepsilon(\nabla^{m+k} \tilde{u}_0 )\|_{L^2(\Omega_{5\varepsilon}\setminus \Omega_{2\varepsilon})}.
\end{align}
In view of  (\ref{pl311}), (\ref{pl312}), (\ref{pl318}) and (\ref{pl319}), we
have proved (\ref{l3.1re1}).

To derive (\ref{l3.1re2}), let us examine the four terms in the RHS of (\ref{l3.1re1}).
 Thanks to Lemma \ref{l2.5} and (\ref{l2.4re1}) in Lemma \ref{l2.4}, we have
\begin{align}\label{pl320}
\begin{split}
 &\| \nabla^m u_0\|_{L^{2}(\Omega_{2\varepsilon})} \leq C \varepsilon^{1/2} \|u_0\|_{H^{m+1}(\Omega)}, \\
& \| \nabla^m \tilde{u}_0  -S_\varepsilon (\nabla^m \tilde{u}_0) \|_{L^2(\Omega\setminus \Omega_{3\varepsilon})} \leq C  \varepsilon \|  \nabla^{m+1} \tilde{u}_0\|_{L^2(\mathbb{R}^d)}\leq C\varepsilon \|u_0\|_{H^{m+1}(\Omega)}.
\end{split}
\end{align}
By Lemmas \ref{l2.3} and \ref{l2.5}, we see that
\begin{align}\label{pl321}
\sum_{0\leq k\leq m-1}\varepsilon^{k } & \|  S_\varepsilon(\nabla^{m+k}  \tilde{u}_0 )\|_{L^2(\Omega_{5\varepsilon}\setminus \Omega_{2\varepsilon})}
 \leq  C \|   \nabla^m  {u}_0   \|_{L^2(\Omega_{6\varepsilon}\setminus \Omega_{\varepsilon})}
 \leq C\varepsilon^{1/2} \|u_0\|_{H^{m+1}(\Omega)}.
 \end{align}
Finally, Lemma \ref{l2.3} implies that
\begin{align}\label{pl322}
\sum_{0\leq k\leq m-1}\varepsilon^{m-k }     \|  S_\varepsilon(\nabla^{2m-k}  \tilde{u}_0    )\|_{L^2(\Omega\setminus \Omega_{2\varepsilon})} &\leq\sum_{0\leq k\leq m-1}\varepsilon   \|  S_\varepsilon(\nabla^{m+1}  \tilde{u}_0    )\|_{L^2(\Omega\setminus \Omega_{\varepsilon})}  \nonumber\\
 &\leq C \varepsilon   \|u_0\|_{H^{m+1}( \Omega)}.
\end{align}
 In view of (\ref{pl320})--(\ref{pl322}) and (\ref{l3.1re1}),  we obtain (\ref{l3.1re2}).
\end{proof}

\begin{theorem}\label{t3.1}
Let $\Omega$ be a bounded Lipschitz domain in $\mathbb{R}^d$.
Assume that the matrix $ A$ satisfies
 (\ref{cod1})--(\ref{cod3}). Let $u_\varepsilon$ $(\varepsilon\geq0)$ be the weak solution
 to  Diricchlet problem (\ref{eq1}) with $\dot{g}=\{g_\gamma\}_{|\gamma|\leq m-1}
 \in W\!A^{m,2}(\partial\Omega;\mathbb{R}^n) , f\in W^{-m+1, p_0}(\Omega;\mathbb{R}^n)$
 and $ p_0=\frac{2d}{d+1}.$
Then
\begin{align}\label{t3.1re1}
&\| w_\varepsilon\|_{H^m_0(\Omega)}\leq C\varepsilon^{1/2} \|u_0 \|_{H^{m+1}(\Omega)},
\end{align} where $w_\varepsilon$ is defined in (\ref{w}).
If in addition $ A$ is symmetric, i.e. $A=A^*$, then
\begin{align}\label{t3.1re2}
 \| w_\varepsilon\|_{H^m_0(\Omega)} \leq C \varepsilon^{1/2}  \left\{\|\dot{g}\|_{W\!A^{m,2}(\partial\Omega)}+\|f\|_{W^{-m+1, p_0}(\Omega)}\right\}.
\end{align}
\end{theorem}

\begin{proof}
 Note that (\ref{t3.1re1}) is a consequence of (\ref{l3.1re2}) by taking  $\phi= w_\varepsilon$.
To prove (\ref{t3.1re2}),
we let $\phi= w_\varepsilon$ in (\ref{l3.1re1}).
 The coercivity condition (\ref{cod2})  implies that
\begin{align}\label{pt311}
\| w_\varepsilon\|_{H^m_0(\Omega)}
&\leq C     \| \nabla^m u_0\|_{L^{2}(\Omega_{2\varepsilon})}
 +  C   \| \nabla^m \tilde{u}_0  -S_\varepsilon(\nabla^m \tilde{u}_0) \|_{L^2(\Omega\setminus \Omega_{2\varepsilon})}\nonumber\\
&\quad + C   \sum_{ 0\leq k \leq m-1,  }\varepsilon^{k } \| S_\varepsilon(\nabla^{m+k} \tilde{u}_0 )\|_{L^2(\Omega_{5\varepsilon}\setminus \Omega_{2\varepsilon})}\nonumber\\
&\quad +  C \sum_{0\leq k\leq m-1}\varepsilon^{m-k } \|  S_\varepsilon(\nabla^{2m-k}  \tilde{u}_0    )\|_{L^2(\Omega\setminus \Omega_{2\varepsilon})}  \nonumber\\
&\doteq J_1+J_2+J_3+J_4.
\end{align}
To bound $J_i, i=1,2,...4 $ by $\|\dot{g}\|_{W\!A^{m,2}(\partial\Omega)}$ and $\|f\|_{W^{-m+1, p_0}(\Omega)}$,
we first note that for a functional $f\in W^{-m+1, p_0}(\Omega; \mathbb{R}^n) $,
there exists an array of functions $f^\zeta\in L^{p_0}(\Omega; \mathbb{R}^n) $ ($\zeta$ is a multi-index) such that
\begin{align*}
f=\sum_{|\zeta|\leq m-1} D^\zeta f^\zeta ~~\text{ and } ~~ \|f\|_{W^{-m+1, p_0}(\Omega)} \approx
 \sum_{|\zeta|\leq m-1} \|f^\zeta\|_{L^{p_0}(\Omega)}.
\end{align*}
Also note that there is a matrix of fundamental solutions $\Gamma^{0, A}(x)$
(with pole at the origin) for the homogenized operator $\mathcal{L}_0$ in $\mathbb{R}^d$ \cite{h,jo,ab},
such that
\begin{equation*}
|D^\eta \Gamma^{0, A}(x)|\leq
\begin{cases} \frac{C_\eta}{|x|^{d-2m+|\eta|}}, \text{ if either } d \text{ is odd, or } d>2m, \text{ or if } |\eta|>2m-d, \vspace{0.3cm}\\
\frac{C_\eta(1+ |\log|x||)}{|x|^{d-2m+|\eta|}}, \text{ if } d \text{ is even},
2\le d\le 2m \text{ and } 0\leq  |\eta|\leq 2m-d,
\end{cases}
\end{equation*}
for any multi-index $\eta.$
Set
\begin{align}\label{pt312}
 v_0(x)=\int_{\mathbb{R}^d} \sum_{|\zeta|\leq m-1} D^\zeta \Gamma^{0, A} (x-y) \widetilde{f}^\zeta(y) \,
 dy  \quad \text{ and } \quad u_0(x)=v_0(x)+v(x),
  \end{align}
  where $ \widetilde{f}^\zeta $ is the extension of $f^\zeta$, being zero outside $\Omega.$
Thanks to the Calder\'{o}n-Zygmund estimates for singular integral
and fractional integral estimates (see e.g. \cite{st} Chapters II, V) we have
\begin{align}
\|\nabla^{m+1} v_0\|_{L^{p_0}(\widehat{\Omega})}\leq C\sum_{|\zeta|\leq m-1}
 \|f^\zeta\|_{L^{p_0}(\Omega)}\leq  C \|f\|_{W^{-m+1, p_0}(\Omega)},   \label{pt313} \\
\|\nabla^sv_0\|_{L^{q_0}(\widehat{\Omega})}\leq C\sum_{|\zeta|\leq m-1} \|f^\zeta\|_{L^{p_0}(\Omega)}
\leq C \|f\|_{W^{-m+1, p_0}(\Omega)},  \label{pt314}
\end{align}
where $\widehat{\Omega} = \{x\in \mathbb{R}^d: dist(x,\Omega)<2\}$, $ \frac{1}{q_0}=1-\frac{1}{p_0}=\frac{2d}{d-1}$ and
 $0\le s\le m$.

Let $\nu$ denote the unit outward normal to $\partial\Omega$.
Let $e=(e_1,...,e_d)\in C_c^\infty( \mathbb{R}^d; \mathbb{R}^d)$ such that $e\cdot \nu\geq c_0>0$ on $\partial\Omega$. Using the divergence theorem, we deduce that for any multi-index $\gamma$ with $0\leq |\gamma|\leq m$,
\begin{align}\label{pt315}
 c_0\int_{\partial\Omega} |D^\gamma v_0|^2 \, d\sigma
 &\leq \int_{\partial\Omega} |D^\gamma v_0|^2 \, \nu\cdot e \, d\sigma\nonumber \\
&\leq \int_{\Omega} | D^{\gamma} v_0 |^2 |\, \text{\rm div}
 (e)| \, dx+ 2 \int_{\Omega} |D_{x_i} D^\gamma  v_0| |e_i| \, | D^\gamma v_0  | \, dx\nonumber \\
&\leq C  \|\nabla^{|\gamma|} v_0\|^2_{L^{p_0}(\Omega)}
+ C\|\nabla^{|\gamma|+1} v_0\|_{L^{p_0}(\Omega)}\|\nabla^{|\gamma|} v_0 \|_{L^{q_0}(\Omega)}\nonumber \\
&\leq    C \|f\|_{W^{-m+1, p_0}(\Omega)},
\end{align}
where we also have used H\"{o}lder's inequality, (\ref{pt313}) and (\ref{pt314}).
Substituting $\partial \Omega$ in  (\ref{pt315})
by $ \Sigma_t=\{x\in \mathbb{R}^d: dist(x,\partial\Omega)=t\}$ for $0<t<<1,$
and integrating the resulting inequality with respect to $t$ from $0$ to $\varepsilon$, we then obtain that
\begin{align}\label{pt316}
\|\nabla^{s}v_0\|_{L^2(\Omega_\varepsilon)}\leq C\varepsilon^{1/2}
 \|f\|_{W^{-m+1, p_0}(\Omega)} ~~\text{for } 0\leq s\leq m.
\end{align}
  Denote $\{Tr(D^\gamma v)\}_{|\gamma|\leq m-1} =\{v_\gamma\}_{|\gamma|\leq m-1}$ as $\dot{v}$. By (\ref{pt315}), we get
\begin{align*}
\|\dot{v}\|_{W\!A^{m,2}(\partial\Omega)}&\leq \|\dot{g}\|_{W\!A^{m,2}(\partial\Omega)}+ \|\dot{v}_0\|_{W\!A^{m,2}(\partial\Omega)}  \\
& \leq \|\dot{g}\|_{W\!A^{m,2}(\partial\Omega)}+ C \|f\|_{W^{-m+1, p_0}(\Omega)}.
\end{align*}

Note that
$$
\mathcal{L}_0 v=\mathcal{L}_0u_0-\mathcal{L}_0v_0=0 \quad \text{ in } \Omega.
$$
Since $A^*=A$, we have $(\bar{A})^*=\overline{A^*}=\bar{A}$.
This allows us to apply the nontangential maximal function estimates for higher-order elliptic systems
with constant coefficients  in Lipschitz domains \cite{pv, v96}
to obtain
\begin{align}\label{pt317}
\|  \mathcal{M}(\nabla^m v)\|_{L^2(\partial \Omega)}
 &\leq C\|\dot{v}\|_{W\!A^{m,2}(\partial\Omega)}\nonumber\\
 &\leq C\|\dot{g}\|_{W\!A^{m,2}(\partial\Omega)}+ C \|f\|_{W^{-m+1, p_0}(\Omega)},
\end{align}
where $\mathcal{M}(\nabla^m v) $ denotes the nontangential maximal function of $\nabla^m v$.
By combining (\ref{pt316}) and (\ref{pt317}), we see that
\begin{align}\label{pt318}
J_1=\|\nabla^mu_0\|_{L^2(\Omega_\varepsilon)} \leq C \varepsilon^{1/2}
 \left\{  \|\dot{g}\|_{W\!A^{m,2}(\partial\Omega)}+   \|f\|_{W^{-m+1, p_0}(\Omega)}\right\}.
\end{align}

Now let us turn to $J_2.$
Let $\widetilde{\rho}_\varepsilon$ be a function in $C_c^\infty(\Omega)$
such that supp$(\widetilde{\rho}_\varepsilon)\subset \Omega\setminus \Omega_{\varepsilon/2}$ and
\begin{align*}
 0\leq\widetilde{\rho}_\varepsilon\leq 1, ~ |\nabla \widetilde{\rho}_\varepsilon|
 \leq C \varepsilon^{-1},~\widetilde{\rho}_\varepsilon=1  \text{ on } \Omega\setminus \Omega_{\varepsilon}.
\end{align*}
Let $\widehat{\rho}_1$ be a function in $C_c^\infty(\widehat{\Omega})$
such that
\begin{align*}
 0\leq \widehat{\rho}_1\leq 1, ~\widehat{\rho}_1(x)=1  \text{ in } \Omega, ~ |\nabla^k  \widehat{\rho}_1|
 \leq C ,~\text{ for } 1\leq k \leq 2m.
\end{align*}
It follows that
\begin{align}\label{pt319}
J_2&\leq \| \nabla^mv_0  -S_\varepsilon(\nabla^m v_0 )  \|_{L^2(\Omega\setminus\Omega_{2\varepsilon} )}+\| \nabla^mv  -S_\varepsilon(\nabla^m v ) \|_{L^2(\Omega \setminus\Omega_{2\varepsilon})} \nonumber\\
&\leq \| \nabla^mv_0\widehat{\rho}_1-S_\varepsilon(\nabla^m v_0\widehat{\rho}_1)\|_{L^2(\mathbb{R}^d)}
+\|  \nabla^m v \widetilde{\rho}_\varepsilon -S_\varepsilon(  \nabla^m v \widetilde{\rho}_\varepsilon )
  \|_{L^2(\mathbb{R}^d)}\nonumber\\
& \leq       C\varepsilon^{1/2}\|\nabla(\nabla^mv_0\widehat{\rho}_1)\|_{L^{p_0}(\mathbb{R}^d)}
+C  \varepsilon \|\nabla( \nabla^m v\widetilde{\rho}_\varepsilon ) \|_{L^2(\mathbb{R}^d)}\nonumber\\
& \leq       C\varepsilon^{1/2}\left\{\|\nabla^{m+1}v_0\|_{L^{p_0}(\widehat{\Omega})} +\|\nabla^{m}v_0\|_{L^{p_0}(\widehat{\Omega})}\right\} + C  \varepsilon\|\nabla^{m+1} v \|_{L^2(\Omega \setminus\Omega_{\varepsilon/2})} + C\| \nabla^m v\|_{L^2(\Omega_{\varepsilon})}\nonumber\\
&\leq C \varepsilon^{1/2} \left\{  \|\dot{g}\|_{W\!A^{m,2}(\partial\Omega)}+  \|f\|_{W^{-m+1, p_0}(\Omega)}\right\}+C  \varepsilon \|\nabla^{m+1} v \|_{L^2(\Omega \setminus\Omega_{\varepsilon/2})},
\end{align}
where we have used (\ref{l2.4re4}), (\ref{l2.4re1}) for the third inequality
as well as (\ref{pt313}), (\ref{pt314}) and (\ref{pt317}) for the last inequality.

For $J_3$ and $J_4$, we observe that by Lemma \ref{l2.3},
\begin{align}\label{pt320}
J_3\leq  C   \|  \nabla^{m  }  {u}_0 \|_{L^2(\Omega_{5\varepsilon})}
\leq C \varepsilon^{1/2} \left\{  \|\dot{g}\|_{W\!A^{m,2}(\partial\Omega)}+ \|f\|_{W^{-m+1, p_0}(\Omega)}\right\}.
\end{align}
Similar to the deduction of (\ref{pt319}),  we use (\ref{l2.4re2}), Lemma \ref{l2.3} as well as (\ref{pt313}) and (\ref{pt314}) to deduce that
\begin{align}\label{pt321}
J_4
&\leq C  \sum_{0\leq k\leq m-1}\varepsilon^{m-k }   \|  S_\varepsilon(\nabla^{2m-k} (v_0 \widehat{\rho}_1)   )\|_{L^2(\Omega\setminus \Omega_{2\varepsilon})} + C  \sum_{0\leq k\leq m-1}\varepsilon^{m-k }   \|
S_\varepsilon(\nabla^{2m-k}  v    )\|_{L^2(\Omega\setminus \Omega_{2\varepsilon})} \nonumber\\
  &\leq C \varepsilon^{1/2}  \| \nabla^{m+1}  (v_0 \widehat{\rho}_1) \|_{L^{p_0}(\widehat{\Omega})}
  + C\varepsilon \|   \nabla^{m+1}  v \|_{L^2(\Omega\setminus \Omega_{\varepsilon})}\nonumber\\
  &\leq C \varepsilon^{1/2}\|f\|_{W^{-m+1, p_0}(\Omega)}+ C\varepsilon \|   \nabla^{m+1}  v \|_{L^2(\Omega\setminus \Omega_{\varepsilon})}.
\end{align}
By combining the estimates for $J_1, J_2,J_3,J_4$ and (\ref{pt311}), we obtain that
 \begin{align}\label{pt322}
\| w_\varepsilon\|_{H^m_0(\Omega)}\leq C \varepsilon^{1/2} \left\{
\|\dot{g}\|_{W\!A^{m,2}(\partial\Omega)}+ \|f\|_{W^{-m+1, p_0}(\Omega)}\right\}
+C  \varepsilon \|\nabla^{m+1} v \|_{L^2(\Omega \setminus\Omega_{\varepsilon/2})}.
 \end{align}
Thus, to prove (\ref{t3.1re2}),  it remains only to bound $ \|\nabla^{m+1} v\|_{L^2(\Omega\setminus \Omega_{\varepsilon})} $.
Recall that $ \mathcal{L}_0 v=\mathcal{L}_0 (u_0-v_0)=0$ in $\Omega$.
By the interior estimates for elliptic systems with constant coefficients,
\begin{align*}
|\nabla^{m+1} v(x)|\leq \frac{C}{\delta(x)} \left\{ \fint_{B(x, \frac{\delta(x)}{8})} |\nabla^{m} v |^2\right\}^{1/2},
\end{align*}
where $\delta(x)=\text{dist} (x, \partial\Omega)$.
This leads to
\begin{align*}
\|\nabla^{m+1} v \|_{L^2(\Omega\setminus \Omega_{\varepsilon/2})}
&\leq   C   \left\{\int_{\Omega\setminus \Omega_{\varepsilon/2}}\frac{1}{(\delta(x))^2}
\fint_{B(x, \frac{\delta(x)}{8})} |\nabla^{m} v(y)|^2dy dx\right\}^{1/2}\nonumber\\
&\le C\varepsilon^{-1/2} \| \nabla^m v\|_{L^2(\Omega)}\nonumber\\
& \leq C \varepsilon^{-1/2} \left\{  \|\dot{g}\|_{W\!A^{m,2}(\partial\Omega)}+  \|f\|_{W^{-m+1, p_0}(\Omega)}\right\},
\end{align*}
which, combined with (\ref{pt322}), implies (\ref{t3.1re2}). The proof is completed.
\end{proof}

With Theorem \ref{t3.1} at our disposal, we are now ready to prove Theorem \ref{tcon}.

\begin{proof}[\bf Proof of Theorem \ref{tcon}]
In view of (\ref{w}) it is enough to prove that
\begin{align}
\|\varepsilon^m\sum_{|\gamma|=m}\chi^\gamma(\frac{x}{\varepsilon}) S^2_\varepsilon(D^\gamma \tilde{u}_0)\rho_\varepsilon\|_{W_0^{m-1, q_0} (\Omega)} \leq C\varepsilon\|u_0\|_{H^{m+1}(\Omega)},\label{ptc001}\\
\|w_\varepsilon\|_{W_0^{m-1, q_0}  (\Omega)}\leq C \varepsilon\|u_0\|_{H^{m+1}(\Omega)}.\label{ptc002}
\end{align}
Note that
\begin{align*}
 &\varepsilon^m\|\sum_{|\gamma|=m}\chi^\gamma(\frac{x}{\varepsilon}) S^2_\varepsilon(D^\gamma \tilde{u}_0)\rho_\varepsilon\|_{W_0^{m-1, q_0} (\Omega)}\\
&\le C  \varepsilon^{m} \sum_{|\gamma|=m}
 \sum_{|\eta_1+\eta_2+\eta_3|=m-1}
 \varepsilon^{-|\eta_1|}
 \| ( D^{\eta_1}\chi^\gamma)(\frac{x}{\varepsilon}) S^2_\varepsilon(D^{\eta_2}D^\gamma \tilde{u}_0)D^{\eta_3}\rho_\varepsilon\|_{L^{q_0} (\Omega)}\\
&\leq C \sum_{0\leq k\leq m-1}\varepsilon^{1+k}  \| S_\varepsilon(\nabla^{m+k} \tilde{u}_0)\|_{L^{q_0}(\Omega)},
\end{align*}
where we have used Lemma \ref{l2.2} and the definition of $\rho_\varepsilon$  for the last inequality.
Using Sobolev imbeddings and Lemma \ref{l2.3}, we obtain that
\begin{align*}
 \varepsilon^m\|\sum_{|\gamma|=m}\chi^\gamma(\frac{x}{\varepsilon}) S^2_\varepsilon(D^\gamma \tilde{u}_0)\rho_\varepsilon\|_{W_0^{m-1, q_0} (\Omega)}
&\leq
C \sum_{0\leq k\leq m-1}\varepsilon^{1+k}  \| S_\varepsilon(\nabla^{m+k+1} \tilde{u}_0)\|_{L^2(\mathbb{R}^d)},\\
& \leq C\varepsilon \|  u_0 \|_{H^{m+1}(\Omega)},
\end{align*}
which gives (\ref{ptc001}).

 Next we turn to (\ref{ptc002}).   For any fixed $F\in W^{-m+1, p_0}(\Omega;\mathbb{R}^n)$,
  let $\psi_\varepsilon\in H^m_0(\Omega;\mathbb{R}^n)$  be the weak solution to the Dirichlet problem
\begin{equation*}
 \begin{cases}
 \mathcal{L}_\varepsilon \psi_\varepsilon =F  &\text{ in } \Omega,  \vspace{0.3cm}\\
 Tr (D^\gamma \psi_\varepsilon)=0,  & \text{ on } \partial\Omega \ \ \text{ for  } 0\leq|\gamma|\leq m-1,
\end{cases}
\end{equation*}
and  $\psi_0\in H^m_0(\Omega;\mathbb{R}^n)$ the solution to the homogenized problem
 \begin{equation*}
 \begin{cases}
 \mathcal{L}_0 \psi_0 =F  &\text{ in } \Omega,  \vspace{0.3cm}\\
 Tr (D^\gamma \psi_0)=0  & \text{ on } \partial\Omega \ \  \text{ for  } 0\leq|\gamma|\leq m-1.
\end{cases}
\end{equation*}
Set $$ \Psi_\varepsilon=\psi_\varepsilon-\psi_0-\varepsilon^m\sum_{|\gamma|=m}\chi^\gamma(\frac{x}{\varepsilon}) S^2_\varepsilon(D^\gamma \tilde{\psi}_0 )\rho_\varepsilon .$$
Since $w_\varepsilon\in H^m_0(\Omega;\mathbb{R}^n)$,  we  deduce that
\begin{align}\label{ptc003}
&  \langle w_\varepsilon, F \rangle_{W_0^{m-1, q_0}(\Omega)\times W^{-m+1, p_0}(\Omega) } =   \sum_{|\alpha|=|\beta|=m}\int_{\Omega} D^\alpha w_\varepsilon A^{\alpha \beta}(\frac{x}{\varepsilon})D^\beta  \psi_\varepsilon  \nonumber\\
 &=   \sum_{|\alpha|=|\beta|=m} \int_{\Omega} D^\alpha w_\varepsilon A^{\alpha \beta}(\frac{x}{\varepsilon})
 D^\beta  \Psi_\varepsilon
 + \sum_{|\alpha|=|\beta|=m}\int_{\Omega} D^\alpha w_\varepsilon
 A^{\alpha \beta}(\frac{x}{\varepsilon})D^\beta\psi_0  \nonumber\\
 &\quad + \sum_{|\alpha|=|\beta|=m}\int_{\Omega}
 D^\alpha w_\varepsilon A^{\alpha \beta}(\frac{x}{\varepsilon})D^\beta
 \left\{\sum_{|\gamma|=m}\varepsilon^m\chi^\gamma(\frac{x}{\varepsilon}) S^2_\varepsilon(D^\gamma \tilde{\psi}_0)\rho_\varepsilon  \right\} \nonumber \\
 &\doteq K_1+K_2+K_3.
\end{align}
By (\ref{t3.1re1}) and (\ref{t3.1re2}), it is easy to see that
\begin{align*}
|K_1|\leq C\| w_\varepsilon\|_{H^m_0(\Omega)}\| \Psi_\varepsilon\|_{H^m_0(\Omega)}\leq C\varepsilon  \|u_0 \|_{H^{m+1}(\Omega)}
   \|F\|_{W^{-m+1, p_0}(\Omega)}.
\end{align*}
Also note that $  \psi_0\in H^m_0(\Omega;\mathbb{R}^n),$ and
 $$
 \aligned
& \|\nabla^m \psi_0 \|_{L^2(\Omega)} \leq C \|F\|_{H^{-m}(\Omega)}\leq C \|F\|_{W^{-m+1, p_0}(\Omega)}, \\
& \| \nabla^m \psi_0 \|_{L^2(\Omega_\varepsilon)}
\le C \, \varepsilon^{1/2} \| F\|_{W^{-m+1,p_0}(\Omega)},
\endaligned
 $$
 where the last inequality was established in the proof of Theorem \ref{t3.1}.
 Hence,  by   (\ref{l3.1re2}),
\begin{align*}
|K_2|
  &\leq C \|u_0\|_{H^{m+1}(\Omega)}\Big\{\varepsilon \|\nabla^m \psi_0 \|_{L^2(\Omega)}+ \varepsilon^{1/2} \|\nabla^m \psi_0\|_{L^2(\Omega_\varepsilon)}  \Big\}\nonumber\\
&\leq C \varepsilon \|u_0\|_{H^{m+1}(\Omega)}\|F\|_{W^{-m+1, p_0}(\Omega)}.
\end{align*}
Also, by (\ref{l3.1re2}),
 \begin{align}\label{ptc004}
  |K_3 |&\leq C \varepsilon\|u_0\|_{H^{m+1} (\Omega)}
  \| \varepsilon^m\nabla^m\Big\{ \chi(\frac{x}{\varepsilon})
  S^2_\varepsilon(\nabla^m \tilde{\psi}_0 )\rho_\varepsilon\Big\}\|_{L^2(\Omega)} \nonumber\\
   &+ C \varepsilon^{1/2} \|u_0\|_{H^{m+1}(\Omega)}\| \varepsilon^m\nabla^m \Big\{ \chi(\frac{x}{\varepsilon}) S^2_\varepsilon(\nabla^m \tilde{\psi}_0   )\rho_\varepsilon\Big\}\|_{L^2(\Omega_{4\varepsilon}
   \setminus \Omega_{3\varepsilon})}.
\end{align}
Observe  that
\begin{align}\label{ptc005}
&\varepsilon^m\|\nabla^m\Big\{ \chi^\gamma(\frac{x}{\varepsilon})
S^2_\varepsilon(\nabla^m \tilde{\psi}_0)\rho_\varepsilon\Big\}\|_{L^2(\Omega)}\nonumber\\
&\leq C \|\nabla^m  \chi^\gamma(\frac{x}{\varepsilon}) S^2_\varepsilon(\nabla^m \tilde{\psi}_0)\rho_\varepsilon\|_{L^2(\Omega)}
 \nonumber\\
 &\ \ +C \varepsilon^m \|\chi^\gamma(\frac{x}{\varepsilon}) S^2_\varepsilon(\nabla^{2m} \tilde{\psi}_0)\rho_\varepsilon\|_{L^2(\Omega)}
+ C\varepsilon^m \|\chi^\gamma(\frac{x}{\varepsilon}) S^2_\varepsilon(\nabla^{m} \tilde{\psi}_0)\nabla^{m}\rho_\varepsilon\|_{L^2(\Omega)}\nonumber\\
 &\ \ +C\sum_{k_1+k_2+k_3=m, k_i\geq1,i=1,2,3}
 \varepsilon^{k_2+k_3}
 \|(\nabla^{k_1}\chi^\gamma)(\frac{x}{\varepsilon}) S^2_\varepsilon(\nabla^{k_2+m} \tilde{\psi}_0 )\nabla^{k_3}\rho_\varepsilon\|_{L^2(\Omega)}\nonumber\\
 &\doteq K_{31}+K_{32}+K_{33} +K_{34}.
\end{align}
For $ K_{31}, K_{33}$,  we deduce from Lemma \ref{l2.2} that
 \begin{align*}
K_{31}\leq C  \|  S_\varepsilon(\nabla^m \tilde{\psi}_0) \|_{L^2(\Omega\setminus \Omega_{2\varepsilon}))}
\leq C\|   \nabla^m {\psi}_0 \|_{L^2(\Omega)} \leq C \|F\|_{W^{-m+1, p_0}(\Omega)},\\
K_{33}\leq C \|  S_\varepsilon(\nabla^m \tilde{\psi}_0 )\|_{L^2(\Omega_{4\varepsilon}\setminus \Omega_{3\varepsilon})}
\leq  C \|   \nabla^m \tilde{\psi}_0 \|_{L^2(\Omega)}   \leq C \|F\|_{W^{-m+1, p_0}(\Omega)}.
\end{align*}
Furthermore, by Lemmas  \ref{l2.2} and  \ref{l2.3}, we see that
 \begin{align*}
K_{32}&\leq C  \varepsilon^m\| \chi^\gamma(\frac{x}{\varepsilon}) S^2_\varepsilon(\nabla^{2m} \tilde{\psi}_0) \|_{L^2(\Omega\setminus \Omega_{3\varepsilon}))}
\leq C  \varepsilon^m \|   S_\varepsilon(\nabla^{2m} \tilde{\psi}_0) \|_{L^2(\Omega)}\\&\leq  C \|F\|_{W^{-m+1, p_0}(\Omega)},
 \end{align*}
 and also
 \begin{align*}
K_{34} &= C\sum_{k_1+k_2+k_3=m, k_i\geq1,i=1,2,3}\varepsilon^{k_2+k_3}
 \| (\nabla^{k_1}\chi^\gamma)(\frac{x}{\varepsilon})  S^2_\varepsilon(\nabla^{k_2+m}
 \tilde{\psi}_0) \nabla^{k_3}\rho_\varepsilon\|_{L^2(\Omega_{4\varepsilon}\setminus\Omega_{2\varepsilon})}\\
& \leq C\sum_{1\leq k_2\leq m-2} \varepsilon^{k_2 }\|  S_\varepsilon(\nabla^{k_2+m}\tilde{\psi}_0)\|_{L^2(\Omega_{4\varepsilon}\setminus\Omega_{\varepsilon})}\\
&\leq C \|  \nabla^{m} \psi_0 \|_{L^2(\Omega_{4\varepsilon})}\leq C \|F\|_{W^{-m+1, p_0}(\Omega)}.
 \end{align*}
By combining the estimates on $K_{31}, K_{32}, K_{33}, K_{34}$ with (\ref{ptc005}), we obtain
 \begin{align}\label{ptc006}
\varepsilon^m\|\nabla^m\Big\{ \chi^\gamma(\frac{x}{\varepsilon}) S^2_\varepsilon(\nabla^m \tilde{\psi}_0)\rho_\varepsilon\Big\}\|_{L^2(\Omega)} \leq C   \|F\|_{W^{-m+1, p_0}(\Omega)}.
 \end{align}
Similar consideration also shows that
\begin{align}\label{ptc007}
\varepsilon^m\|\nabla^m \Big\{ \chi^\gamma(\frac{x}{\varepsilon}) S^2_\varepsilon(\nabla^m \tilde{\psi}_0   )\rho_\varepsilon\Big\}\|_{L^2(\Omega_{4\varepsilon}\setminus \Omega_{3\varepsilon})}
\leq C \varepsilon^{1/2}  \|F\|_{W^{-m+1, p_0}(\Omega)}.
\end{align}
By combining (\ref{ptc006}) and (\ref{ptc007}) with (\ref{ptc004}), we obtain that
\begin{align*}
 |K_3|\leq C \varepsilon\|u_0\|_{H^{m+1}(\Omega)}   \|F\|_{W^{-m+1, p_0}(\Omega)} .
 \end{align*}
Finally, in view of the estimates of $K_1, K_2, K_3$ and (\ref{ptc003}), we have proved that
 \begin{align*}
 \Big| \langle w_\varepsilon, F \rangle_{W_0^{m-1, q_0}(\Omega)\times W^{-m+1, p_0}(\Omega) } \big|
 \leq C \varepsilon\|u_0\|_{H^{m+1}(\Omega)}  \|F\|_{W^{-m+1, p_0}(\Omega)},
  \end{align*}
  which, by duality,  gives the desired estimate (\ref{ptc002}).
  This completes the proof of Theorem \ref{tcon}.
  \end{proof}

\begin{remark}\label{rsymm}
The symmetric assumption on $A$ is made to ensure the nontangential
maximal function estimates (\ref{pt317}) in Lipschitz domains.
If $\Omega$ is smooth, this assumption can be removed
without changing the results in (\ref{t3.1re1}) and (\ref{t3.1re2}), see e.g. \cite{pv,v96}.
Especially,  we still have the following estimate as a consequence of (\ref{t3.1re2}),
\begin{align}\label{rsc}
 \| u_\varepsilon-u_0\|_{L^2(\Omega)} \leq C \varepsilon^{1/2}  \left\{\|\dot{g}\|_{WA^{m,2}(\partial\Omega)}+\|f\|_{W^{-m+1, p_0}(\Omega)}\right\},
\end{align}
which will play an essential role in  the proof of the large-scale $C^{m-1,1}$ estimate in the next section, see Lemma \ref{lemm4.1}.
\end{remark}


\section{$C^{m-1, 1}$ estimates}

This section is devoted to the interior $C^{m-1, 1}$ estimates for $\mathcal{L}_\varepsilon$,
without smoothness and symmetry assumptions on the coefficients.
 The proof is based on a general scheme for establishing
 regularity estimates at large scale in the homogenization theory, formulated in \cite{ak} and further developed in \cite{as, sh1}.

In the following we will always assume that the matrix $ A$ satisfies (\ref{cod1})--(\ref{cod3}).

\begin{lemma}\label{lemm4.1}
For $0<\varepsilon\le r<\infty$, let $B_r=B(x_0,r)$ be a ball in $\mathbb{R}^d$
and  $u_\varepsilon\in H^m(B_{2r}; \mathbb{R}^n)$  a solution to $\mathcal{L}_\varepsilon u_\varepsilon=\sum_{|\alpha|\le m-1} D^\alpha f^\alpha$ in $B_{2r}$.
Then there exists a function $u_0\in H^m(B_{r}; \mathbb{R}^n)$ such that $\mathcal{L}_0 u_0=\sum_{|\alpha|\le m-1} D^\alpha f^\alpha $ in $B_r$ and
\begin{align}\label{l4.1re1}
\left(\fint_{B_r} |u_\varepsilon-u_0|^2\right)^{1/2} \leq C \left(\frac{\varepsilon}{r}\right)^{1/2}     \left\{ \left(\fint_{B_{2r}} |u_\varepsilon|^2\right)^{1/2} +
\sum_{|\alpha|\le m-1} r^{2m-|\alpha|} \left(\fint_{B_{2r}} |f^\alpha |^2\right)^{1/2}  \right\}.
\end{align}
\end{lemma}

\begin{proof}
We may assume that $r=1$ and $x_0=0$ by rescaling and translation.
Let $u_0$ be the weak solution to the Dirichlet problem
\begin{equation}\label{pl4101}
\begin{cases}
\mathcal{L}_0 u_0=\sum_{|\alpha|\le m-1} D^\alpha f^\alpha &\text{ in } B_t,    \\
 Tr (D^\gamma u_0)=D^\gamma u_\varepsilon  & \text{ on } \partial B_t \ \  \text{ for  } 0\leq|\gamma|\leq m-1,
 \end{cases}
\end{equation}
where $t\in [5/4, 7/4]$ is to be determined later.
Thanks to Remark \ref{rsymm}, we have
\begin{align}\label{pl4102}
\|u_\varepsilon-u_0\|_{L^2(B_1)}&\leq   \|u_\varepsilon-u_0\|_{L^2(B_t)}\nonumber\\
 &\leq C \varepsilon^{1/2}\left\{ \sum_{ 0\leq k\leq m}   \| \nabla^k u_\varepsilon \|_{L^2(\partial B_t)} + \sum_{|\alpha|\le m-1}
 \|f^\alpha \|_{L^2(B_t)} \right \}.
 \end{align}
By Caccioppoli's inequality (see Lemma \ref{l2.6}), we obtain that
\begin{align}\label{pl4103}
 \sum_{0\leq k\leq m}\int_{B_\frac{7}{4}}
 | \nabla^k u_\varepsilon|^2 \leq C \left\{\int_{B_2}  |u_\varepsilon|^2 +\sum_{|\alpha|\le m-1} \int_{B_2}  |f^\alpha |^2\right\}.
\end{align}
Hence there must be some $t\in [5/4, 7/4]$ such that
\begin{align}\label{pl4104}
 \sum_{0\leq k\leq m}\int_{\partial B_t} | \nabla^k u_\varepsilon|^2
 \leq C \left\{ \int_{B_2}  |u_\varepsilon|^2 + \sum_{|\alpha| \le m-1}
 \int_{B_2}  |f^\alpha |^2\right\}.
\end{align}
For  otherwise, we may deduce from the co-area formula that
 \begin{align*}
 \sum_{0\leq k\leq m}\int_{B_\frac{7}{4}} | \nabla^k u_\varepsilon|^2
 \ge C \sum_{0\leq k\leq m}\int_\frac{5}{4}^\frac{7}{4}\int_{\partial B_t}
  | \nabla^k u_\varepsilon|^2dt >  C \left\{\int_{B_2}  |u_\varepsilon|^2 + \sum_{|\alpha|\le m-1}
  \int_{B_2}  |f^\alpha |^2\right\},
\end{align*}
which contradicts with (\ref{pl4103}). It follows by (\ref{pl4104}) and (\ref{pl4102}) that
\begin{align}\label{pl4105}
\|u_\varepsilon-u_0\|_{L^2(B_1)}
 \leq C \varepsilon^{1/2}\Big\{   \|u_\varepsilon\|_{L^2(B_2)} + \sum_{|\alpha|\le m-1}
 \|f^\alpha \|_{L^2(B_2)} \Big\}.
 \end{align}
 This completes the proof.
\end{proof}

Let
$$\mathfrak{P }_k=\Big\{(P^1_k, P^2_k,..., P^n_k)\mid  P_k^i \text{ are polynomials of degree at most  } k  \Big\}.
$$

\begin{lemma}\label{lemm4.2}
 Let $B_r=B(x_0,r)$ be a ball in $\mathbb{R}^d$.
 Let $u_0\in H^m(B_{r}; \mathbb{R}^n)$ be a solution to $\mathcal{L}_0 u_0=\sum_{|\alpha |\le m-1} D^\alpha f^\alpha $ in $B_r$ with $f^\alpha \in L^q(B_r; \mathbb{R}^n)$,
 where $q>d$.
 For $0< t \leq r,$ define
\begin{align*}
G(t;u_0)=\frac{1}{t^m} \inf_{P_m\in\mathfrak{P}_{m}} \left\{ \left(\fint_{B_t} |u_0-P_m|^2\right)^{1/2} +
\sum_{|\alpha|\le m-1} t^{2m-|\alpha|} \left(\fint_{B_t} |f^\alpha |^q\right)^{1/q}  \right\}.
\end{align*}
Then there exists some $\delta\in (0, 1/8)$, depending only on $d$, $n$, $m$, $q$ and $\mu$ such that
\begin{align}\label{l42re1}
G(\delta r;u_0)\leq \frac{1}{2} G(r;u_0).
\end{align}
\end{lemma}
\begin{proof}
  By translation and rescaling we may assume that $ x_0=0$ and $r=1$.
 By choosing
 $$
 P_m(x)= \sum_{|\alpha|=0}^{m} \frac{1}{\alpha!} D^\alpha u_0(0) x^\alpha=\sum_{|\alpha|=0}^{m} \frac{1}{\alpha_1!\alpha_2!...\alpha_d!} D^\alpha u_0(0) x_1^{\alpha_1}x_2^{\alpha_2}...x_d^{\alpha_d},
 $$
 we see that
  \begin{align}\label{pl4201}
 G(\delta; u_0) \leq C \delta^\theta  \|\nabla^{m} u_0\|_{C^{0, \theta}(B_\delta)}
 +  \sum_{|\alpha|\le m-1} \delta^{m-|\alpha|} \left(\fint_{B_\delta} |f^\alpha |^q\right)^{1/q}.
  \end{align}
Let $0<\theta<1-\frac{d}{q}$.
It follows from  the $C^{m,\theta}$ regularity for higher-order elliptic systems with constant coefficients
(see e.g.  \cite{gm}) that
\begin{align*}
G(\delta; u_0)
 &\leq C \delta^\theta \left\{ \left(\fint_{B_1} |u_0-P_m|^2\right)^{1/2} + \sum_{|\alpha|\le m-1}
 \left(\fint_{B_1} |f^\alpha |^q\right)^{1/q} \right\}+ C \delta^{1-\frac{d}{q}} \sum_{|\alpha|\le m-1}
 \left(\fint_{B_1} |f^\alpha |^q\right)^{1/q}\nonumber\\
 &\leq C \delta^\theta
  \left\{ \left(\fint_{B_1} |u_0-P_m|^2\right)^{1/2} + \sum_{|\alpha|\le m-1} \left(\fint_{B_1} |f^\alpha |^q\right)^{1/q} \right\}\nonumber
  \end{align*}
  for any $P_m \in \mathfrak{P}_m$. Thus
  $$
 G(\delta; u_0)\leq \frac{1}{2} G(1;u_0),
$$
 if $ \delta \in (0, 1/8)$ is sufficiently small.
 \end{proof}

\begin{lemma}\label{lemm4.3}
For $\varepsilon\in (0, 1/4)$, let $u_\varepsilon$  be
a solution to $\mathcal{L}_\varepsilon u_\varepsilon=\sum_{|\alpha|\le m-1} D^\alpha f^\alpha $ in $B_1$ with $f^\alpha \in L^q(B_1; \mathbb{R}^n)$ for some $q>d $.
For $ 0<r\leq 1/2 $, define
\begin{align*}
H(r)=\frac{1}{r^m} \inf_{P_m\in\mathfrak{P}_{m}} \left\{ \left(\fint_{B_r} |u_\varepsilon-P_m|^2\right)^{1/2}
+ \sum_{|\alpha|\le m-1}
r^{2m-|\alpha|} \left(\fint_{B_r} |f^\alpha |^q\right)^{1/q}  \right\}, \\
I(r)= \frac{1}{r^m} \inf_{P_{m-1}\in\mathfrak{P}_{m-1}} \left\{ \left(\fint_{B_{r}} |u_\varepsilon-P_{m-1}|^2\right)^{1/2} +
\sum_{|\alpha|\le m-1} r^{2m-|\alpha| } \left(\fint_{B_{r}} |f^\alpha |^2\right)^{1/2}  \right\}.
\end{align*}
Then
\begin{align}\label{l43re1}
H(\delta r)\leq \frac{1}{2} H(r)+ C \left(\frac{\varepsilon}{r}\right)^{1/2} I(2r)
\end{align}
for any $r\in[\varepsilon, 1/2]$, where $\delta \in (0,1/8)$ is given by Lemma \ref{lemm4.2}.
\end{lemma}

\begin{proof}
For any fixed $r\in [\varepsilon, 1/2]$, let $u_0$  be a solution to $\mathcal{L}_0 u_0=\sum_{|\alpha|\le m-1} D^\alpha f^\alpha $  in $B_r$. Using Lemma  \ref{lemm4.2}, we deduce that
\begin{align*}
 H(\delta r) &\leq \left(\frac{1}{\delta r}\right)^{m} \left(\fint_{B_{\delta r}} |u_\varepsilon-u_0|^2\right)^{1/2}+G(\delta r; u_0)\\
&\leq \left(\frac{1}{\delta r}\right)^{m} \left(\fint_{B_{\delta r}} |u_\varepsilon-u_0|^2\right)^{1/2} +\frac{1}{2} G(r;u_0)\\
&\leq  \left(\frac{1}{\delta r}\right)^{m} \left(\fint_{B_{\delta r}} |u_\varepsilon-u_0|^2\right)^{1/2}+ \frac{C}{r^m} \left(\fint_{B_{ r}} |u_\varepsilon-u_0|^2\right)^{1/2}+\frac{1}{2} H(r).
\end{align*}
This, together with Lemma \ref{lemm4.1}, implies that
\begin{align*}
H(\delta r)
&\leq C \left(\frac{\varepsilon}{r}\right)^{1/2}  \frac{1}{r^m}   \left\{ \left(\fint_{B_{2r}} |u_\varepsilon|^2\right)^{1/2} +
\sum_{|\alpha|\le m-1}  r^{2m-|\alpha|} \left(\fint_{B_{2r}} |f^\alpha |^2\right)^{1/2}  \right\}+\frac{1}{2} H(r).
\end{align*}
By replacing $u_\varepsilon$ with $u_\varepsilon-P_{m-1}$ for any  $P_{m-1} \in \mathfrak{P}_{m-1}$, we obtain
(\ref{l43re1}).
\end{proof}

The following lemma, a continuous version of Lemma 3.1 in \cite{as}, was first proved in \cite{sh1} (Lemma 8.5 therein). It  plays an essential role in our proof of Theorem \ref{tlip}.

\begin{lemma}\label{lemm4.4}
Let $H(r)$ and $h(r)$ be two nonnegative continuous functions on the interval $(0,1],$ and let $\varepsilon \in (0,1/4).$ Assume that
\begin{align}\label{l44con1}
\max_{r\leq t\leq 2r} H(t)\leq C_0H(2r), ~~~~~\max_{r\leq t,s\leq 2r} | h(t)-h(s)|\leq C_0H(2r),
\end{align}
for any $r\in [\varepsilon, 1/2],$  and also
\begin{align}\label{l44con2}
H(\delta r) \leq \frac{1}{2} H(r) + C_0 \omega (\varepsilon/r)\left\{ H(2r)+h(2r)\right\},
\end{align}
for any $r\in [\varepsilon, 1/2],$ where  $\delta\in (0,1/4)$  and $\omega$ is a nonnegative increasing function on $[0,1]$ such that
$\omega(0)=0$ and
\begin{align}\label{l44con3}
\int_0^1 \frac{\omega(\varsigma)}{\varsigma}d\varsigma< \infty.
\end{align}
Then
\begin{align}\label{l44re1}
\max_{\varepsilon\leq r\leq 1} \left\{H(r) +h(r)\right\}\leq C \left\{H(1) +h(1)\right\}.
\end{align}
\end{lemma}

We are now ready to give the proof of Theorem \ref{tlip}.

\begin{proof}[\bf Proof of Theorem \ref{tlip}]
By translation and rescaling we may assume that $x_0=0$ and $R=1$.
We also assume that  $\varepsilon\in (0, 1/4)$.
For otherwise we have $r\in [1/4, 1/2)$ and the result is trivial.
Let $u_\varepsilon$  be a solution to $\mathcal{L}_\varepsilon u_\varepsilon=\sum_{|\alpha|\le m-1} D^\alpha f^\alpha $ in $B_1$ with $f\in L^q(B_1)$ for some
$ q>d$.
For  $r\in (0,1)$, let $H(r), I(r) $ be defined as in Lemma \ref{lemm4.3} and let $\omega(t)=t^{1/2}$, which satisfies (\ref{l44con3}).
Let
$$ h(r)=  \sum_{|\alpha|=m}\frac{1}{\alpha!} |D^\alpha P_{mr}|,
$$
where
$P_{mr}$ is an element in $ \mathfrak{P}_{m}$  such that
$$H(r)=\frac{1}{r^m}
\left\{ \left(\fint_{B_r} |u_\varepsilon-P_{mr}|^2\right)^{1/2} + \sum_{|\alpha|\le m-1}
r^{2m-|\alpha|} \left(\fint_{B_r} |f^\alpha |^q\right)^{1/q}  \right\}.$$
Next let us verify  that $H(r), h(r)$ satisfy  conditions (\ref{l44con1}) and (\ref{l44con2}).
Since $t\in [r,2r]$, from the definition it is obvious that
\begin{align}\label{ptlip01}
H(t)\leq  C H(2r).
\end{align}
Also, by the definition of $h(r),$ we have
\begin{align*}
|h(t)-h(s)|&\leq \sum_{|\alpha|=m} \frac{1}{\alpha!} |D^\alpha (P_{mt}-P_{ms})|.
\end{align*}
Since $\mathcal{L}_0 (P_{mt} -P_{mr})=0$ in $\mathbb{R}^d$,
 it follows from Caccioppoli's inequality that for any $t, s\in [r, 2r]$,
\begin{equation}\label{ptlip02}
\aligned
|h(t)-h(s)|
& \leq \frac{C}{r^m}
 \left(\fint_{B_r} |  P_{mt}-P_{ms} |^2 \right)^{1/2}\\
 &\leq \frac{C}{r^m}
  \left(\fint_{B_r} | u_\varepsilon -P_{mt} |^2 \right)^{1/2}
  + \frac{C}{r^m} \left(\fint_{B_r} | u_\varepsilon -P_{ms} |^2 \right)^{1/2}\\
 &\leq \frac{C}{r^m}  \left(\fint_{B_t} | u_\varepsilon -P_{mt} |^2 \right)^{1/2}
 + \frac{C}{s^m} \left(\fint_{B_s} | u_\varepsilon -P_{ms}  |^2 \right)^{1/2}\\
 &\leq C \{ H(t) +H(s) \}\\
 &\leq C H(2r).
 \endaligned
\end{equation}
This gives condition (\ref{l44con1}).

To show (\ref{l44con2}), we use (\ref{l43re1}) and the observation  that
$$
\aligned
I(2r) & \le H(2r) + \frac{1}{(2r)^m} \inf_{P_{m-1}\in\mathfrak{P}_{m-1}}
\left(\fint_{B_{2r}} |P_{m(2r)} -P_{m-1}|^2\right)^{1/2}\\
&\le  H(2r) + C\, h(2r),
\endaligned
$$
where the last step follows from Poincar\'e's inequality.
 Therefore, by Lemma \ref{lemm4.4}, we obtain
 \begin{equation}\label{Lip-11}
\aligned
 \frac{1}{r^m}
 \inf_{P_{m-1}\in\mathfrak{P}_{m-1}}    \left(\fint_{B_r} |u_\varepsilon-P_{m-1}|^2\right)^{1/2}
 &\le H(r) +  \frac{1}{r^m} \inf_{P_{m-1}\in\mathfrak{P}_{m-1}}
\left(\fint_{B_{r}} |P_{mr} -P_{m-1}|^2\right)^{1/2}\\
  &\leq C \left\{ H(r)+h(r)  \right\}\\
  &\leq C \left\{ H(1)+h(1)  \right\}  \\
  &\leq C   \left\{ \left(\fint_{B_1} |  u_\varepsilon|^2\right)^{1/2} + \sum_{|\alpha|\le m-1}
  \left(\fint_{B_1} |f^\alpha |^q\right)^{1/q}  \right\},
\endaligned
\end{equation}
for any $r\in (\varepsilon, 1/2)$,
where in the last step we have used the observation that
$$
\aligned
h(r)  &\le \frac{C}{r^m} \left(\fint_{B_r} |P_{mr}|^2\right)^{1/2}\\
& \le \frac{C}{r^m} \left(\fint_{B_r} |u_\varepsilon -P_{mr}|^2\right)^{1/2}
+\frac{C}{r^m} \left(\fint_{B_r} |u_\varepsilon|^2\right)^{1/2}\\
&\le C\,  H(r) + \frac{C}{r^m} \left(\fint_{B_r} |u_\varepsilon|^2\right)^{1/2}.
\endaligned
$$
The desired estimate (\ref{larlip}) now follows from (\ref{Lip-11})
by Caccioppoli's inequality.

We now turn to the second part of Theorem \ref{tlip}.
Again we may assume that $x_0=0$ and $R=1$.
We also assume that $ 0<\varepsilon<(1/2)$; the case $\varepsilon \ge (1/2)$
 follows from the standard $C^{m,\theta}$-regularity results for higher-order elliptic systems \cite{gm}:
 if $A$   satisfies (\ref{cod1})--(\ref{cod3}) and (\ref{hol}), and $u$
 is a solution to $\mathcal{L}_1 u =\sum_{|\alpha|\le m-1} D^\alpha f^\alpha $ in $B_1$ with $f^\alpha \in L^q(B_1; \mathbb{R}^n)$ for some
 $ q>d,$ then
\begin{align}\label{ptlip03}
 |\nabla^m u(0)| \leq C \left\{ \left(\fint_{B_1} |\nabla^m u |^2\right)^{1/2} + \sum_{|\alpha|\le m-1}
 \left(\fint_{B_1} |f^\alpha |^q\right)^{1/q}\right\},
\end{align}
where $C$ depends on $d,n,m, \mu, q, \Lambda_0$ and $\tau_0.$

To handle the case $0<\varepsilon<\frac{1}{2},$ we set $w(x)=\varepsilon^{-m}u_\varepsilon(\varepsilon x).$
Direct computations yield that $ \mathcal{L}_1w=\sum_{|\alpha|\le m-1} D^\alpha \big\{ \varepsilon^{m-|\alpha|}
f^\alpha (\varepsilon x)\big\} $ in $B_1$.
It then follows from (\ref{ptlip03}) that
\begin{align*}
|\nabla^m w(0)| &\leq C \left\{ \left(\fint_{B_1} |\nabla^m w |^2\right)^{1/2}
+ \sum_{|\alpha|\le m-1}
\varepsilon^{m-|\alpha|} \left(\fint_{B_1} |f^\alpha (\varepsilon x)|^qdx\right)^{1/q}\right\}\nonumber\\
&\leq C \left\{ \left(\fint_{B_\varepsilon} |\nabla^m u_\varepsilon |^2\right)^{1/2}
+ \sum_{|\alpha|\le m-1}
\varepsilon^{m-|\alpha|-\frac{d}{q}} \left(\fint_{B_1} |f^\alpha |^qdx\right)^{1/q}\right\}.
\end{align*}
This, combined with (\ref{larlip}) for $r=\varepsilon$
and the fact that $\nabla^m w(0)=\nabla^m u_\varepsilon(0)$,
 gives the estimate (\ref{lip}).
 \end{proof}

As a consequence of (\ref{larlip}), we establish a Liouville type result.
Note that  we only assume $A$ is elliptic and bounded measurable, apart from the periodicity condition (\ref{cod3}).

\begin{theorem}
Assume that $A$ satisfies (\ref{cod1})--(\ref{cod3}).
Let $u\in H^m_{loc}(\mathbb{R}^d; \mathbb{R}^n)$ be a weak solution to
 $ \mathcal{L}_1 u =0  \text{ in }  \mathbb{R}^d. $  Suppose that there exist a constant $C_u>0$ and some $\delta\in (0,1)$
 such that
\begin{align}\label{t41con1}
\left(\fint_{B(0, R)} |u|^2\right)^{1/2} \leq C_u R^{m-1+\delta}\quad
 \text{ for any } R>1.
\end{align}
 Then   $u\in \mathfrak{P}_{m-1}$.
\end{theorem}

\begin{proof}
It follows from (\ref{larlip}) that for $1<r<R/2$,
$$
\aligned
\left(\fint_{B(0, r)} |\nabla^m u |^2\right)^{1/2}
&\le \frac{C}{R^m} \left(\fint_{B(0, R)} | u |^2\right)^{1/2}\\
&\le C R^{\delta-1}.
\endaligned
$$
By letting $R\to \infty$ we see that $\nabla^m u=0$ in $B(0,r)$.
Since $r>1$ is arbitrary, it follows that $\nabla^m u=0$ in $\mathbb{R}^d$.
This implies that each component of $u$ is a polynomial of degree at most $m-1$.
\end{proof}


\section{$W^{m,p}$ estimates}

It follows from (\ref{larlip}) and Poincar\'e's inequality that
if $\mathcal{L}_\varepsilon (u_\varepsilon)=0$ in $B(x_0, r)$ and $0<\varepsilon<r$, then
\begin{equation}\label{Large-scale-lip}
\left(\fint_{B(x_0, \varepsilon)} |\nabla^m u_\varepsilon|^2 \right)^{1/2}
\le C \left(\fint_{B(x_0, r)} |\nabla^m u_\varepsilon|^2 \right)^{1/2},
\end{equation}
where $C$ depends only on $d$, $m$, $n$ and $\mu$.
In this section we will use (\ref{Large-scale-lip}) to establish
 the uniform  $W^{m,p}$ estimates
 under the additional smoothness assumption: $A\in V\!M\!O (\mathbb{R}^d)$.

 \begin{lemma}\label{lemma-5.1}
 Assume that $A$ satisfies  (\ref{cod1})--(\ref{cod3}) and
$A\in V\!M\!O (\mathbb{R}^d)$.
Let $u_\varepsilon \in H^m(2B; \mathbb{R}^n)$ be a weak solution to
$\mathcal{L}_\varepsilon (u_\varepsilon)=0$ in $2B$ for some ball $B=B(x_0, r)$.
Then for any $2<p<\infty$,
\begin{equation}\label{5.1-1}
\left(\fint_B |\nabla u_\varepsilon|^p\right)^{1/p}
\le C_p \left(\fint_{2B} |\nabla u_\varepsilon|^2\right)^{1/2},
\end{equation}
where $C_p$ depends only on $d$, $m$, $n$, $p$, $\mu$ and $\varrho(t)$ in (\ref{vmo}).
\end{lemma}

\begin{proof}
By translation and dilation we may assume that $x_0=0$ and $r=1$.
Note that the case $\varepsilon\ge (1/4)$ follows from the existing regularity results for higher-order
equations with VMO coefficients \cite{dk}. This is because $A(x/\varepsilon)$ satisfies (\ref{vmo}) uniformly in
$\varepsilon$.

To handle the case $0<\varepsilon<(1/4)$,
we let $w(x)=u_\varepsilon (\varepsilon x)$.
Then $\mathcal{L}_1 (w)=0$ in $B(0, 2\varepsilon^{-1})$.
It follows  from \cite{dk} that
$$
\left(\fint_{B(0,1/2)} |\nabla^m w|^p \right)^{1/p}
\le C \left(\fint_{B(0,1)} |\nabla^m w|^2 \right)^{1/2}.
$$
By a change of variables this leads to
$$
\left(\fint_{B(0,\varepsilon/2 )} |\nabla^m u_\varepsilon|^p \right)^{1/p}
  \le C \left(\fint_{B(0,\varepsilon )} |\nabla^m u_\varepsilon|^2 \right)^{1/2}
 \le C \left(\fint_{B(0,2 )} |\nabla^m u_\varepsilon|^2 \right)^{1/2},
$$
where we have used (\ref{Large-scale-lip}) for the last inequality.
The same argument also shows that
\begin{equation}\label{5.1-2}
\left(\fint_{B(y,\varepsilon/2 )} |\nabla^m u_\varepsilon|^p \right)^{1/p}
  \le C \left(\fint_{B(0,2 )} |\nabla^m u_\varepsilon|^2 \right)^{1/2},
\end{equation}
for any $y\in B(0, 1)$. It follows that
\begin{equation}\label{5.1-3}
\int_{B(y, \varepsilon/2)} |\nabla^m u_\varepsilon|^p
\le C \varepsilon^d \| \nabla^m u_\varepsilon\|^p_{L^2(B(0,2))}.
\end{equation}
By covering $B(0,1)$ with a finite number of balls $\{ B(y_i, \varepsilon/2)\}$,
we may deduce (\ref{5.1-1}) from (\ref{5.1-3}).
\end{proof}

Our proof of Theorem \ref{twmp} relies on a real variable argument in the following theorem,
formulated in \cite{szt, szf} (also see  \cite{ci} for the original ideas).

\begin{theorem} \label{theo5.1}
Let $F\in L^2(4B_0)$ and $f\in L^p(4B_0)$ for some $2<p<q<\infty$, where $B_0$ is a ball in $\mathbb{R}^d$. Suppose that for each ball $B\subset 2B_0$ with $|B|<c_0|B_0|$, there exists two measurable functions $F_B$ and $R_B$ on $2B$ such that $|F|\leq |F_B|+|R_B|$ on $2B$, and
\begin{align}
&\left(\fint_{2B} |R_B|^q\right)^{1/q}\leq C_1 \left\{ \left(\fint_{c_2B} |F|^2\right)^{1/2} +\sup_{B\subset B' \subset 4B_0} \left(\fint_{B'} |f|^2\right)^{1/2}\right\},\label{t51con1} \\
  &\left(\fint_{2B} |F_B|^2\right)^{1/2} \leq C_2 \sup_{B\subset B' \subset 4B_0} \left(\fint_{B'} |f|^2\right)^{1/2}, \label{t51con2}
\end{align}
where $C_1,C_2>0$, $0<c_1<1$ and $c_2>2.$  Then $F\in L^p(B_0)$ and
\begin{align}
\left(\fint_{B_0} |F|^p\right)^{1/p} \leq C   \left(\fint_{4B_0} |F|^2\right)^{1/2}+ \left(\fint_{4B_0} |f|^p\right)^{1/p}, \label{t51re1}
\end{align}
where $C$ depends only on $d, C_1,C_2, c_1, c_2, p$ and $ q$.
\end{theorem}

\begin{proof}
See \cite[Theorem 3.2]{szt}.
\end{proof}

\begin{proof}[\bf Proof of Theorem \ref{twmp}.]
Let $u_\varepsilon$ be a solution to
$$
\mathcal{L}_\varepsilon u_\varepsilon=\sum_{|\alpha|\le m} D^\alpha f^\alpha  \quad \text{ in } 2B_0,
$$
  with $f^\alpha\in L^p(2B_0; \mathbb{R}^n)$ for some $2<p<\infty $.
  By rescaling we may assume that diam$(B_0)=2$.
   For each ball $B$ with $4B\subset 2B_0$,
   we decompose  $u_\varepsilon$ as $ u_\varepsilon=v_\varepsilon+w_\varepsilon$ on $2B$,
   where $v_\varepsilon \in H^m_0(4B; \mathbb{R}^n)$
   is the solution to  $\mathcal{L}_\varepsilon v_\varepsilon=\sum_{|\alpha|\le m} D^\alpha f^\alpha  \text{ in } 4B$
   and $w_\varepsilon$ is the solution to $\mathcal{L}_\varepsilon w_\varepsilon= 0  \text{ in } 4B$. Setting $q=p+1$,
   $$
    F=|\nabla^mu_\varepsilon|, \quad F_B=|\nabla^mv_\varepsilon|, \quad
    R_B=|\nabla^mw_\varepsilon| \quad \text{ and } \quad
    f=\sum_{|\alpha|\le m} |f^\alpha|.
    $$
     Clearly, $ F\leq F_B+R_B $ on $2B$.
     Note that (\ref{t51con2}) follows from the standard energy estimates.
     Therefore,  to derive (\ref{wmp}), we only need to verify  condition (\ref{t51con1})
      for any $ 2<p<\infty$.  This is done by using Lemma \ref{lemma-5.1}.
      Indeed,
\begin{align*}
\left(\fint_{2B} |R_B|^q\right)^{1/q}&\leq  C \left(\fint_{2B } |\nabla^m w_\varepsilon|^2\right)^{1/2}\\
&\leq C \left(\fint_{2B} |\nabla^m u_\varepsilon|^2\right)^{1/2}+ \left(\fint_{2B} |\nabla^m v_\varepsilon|^2\right)^{1/2}\\
&\leq  C \left\{ \left(\fint_{4B} |F|^2\right)^{1/2} +  \sum_{|\alpha|\le m}\left(\fint_{4B} |f^\alpha|^2\right)^{1/2}\right\}.
\end{align*}
This completes the proof.
\end{proof}

The interior $W^{m,p}$ estimate in Theorem \ref{twmp} gives the following interior H\"{o}lder  and $L^\infty$ estimates by Sobolev imbedding.

\begin{corollary}\label{cor5.1}
Assume that $A$ satisfies   (\ref{cod1})--(\ref{cod3}) and (\ref{vmo}).
Let $u_\varepsilon\in H^m(2B; \mathbb{R}^n)$ be a weak solution to
$$
 \mathcal{L}_\varepsilon u_\varepsilon=\sum_{|\alpha|\le m} D^\alpha f^\alpha \qquad  \text{ in } 2B ,
 $$
where $B=B(x_0, r)$ and $f^\alpha\in L^p(2B; \mathbb{R}^n)$ for some $p>d $.
 Then
\begin{align}
|\nabla^{m-1} u_\varepsilon(x)-\nabla^{m-1} u_\varepsilon(y)|
& \leq \frac{C}{r^{m-1}}
 \left(\frac{|x-y|}{r}\right)^{\theta}
  \left\{  \left(\fint_{2B}|  u_\varepsilon|^2\right)^{1/2} + \sum_{|\alpha|\le m}r^{2m-|\alpha|}
  \left(\fint_{2B}|f^\alpha|^p\right)^{1/p}   \right\},\label{co51re1}\\
\|\nabla^k u_\varepsilon\|_{L^\infty(B)} & \leq C   r^{-k}  \left\{  \left(\fint_{2B}|u_\varepsilon|^2\right)^{1/2} +
\sum_{|\alpha|\le m}r^{2m-|\alpha|}
  \left(\fint_{2B}|f^\alpha|^p\right)^{1/p}    \right\} , \label{co51re2}
\end{align}
where $0\le k\le m-1$, $\theta=1-\frac{d}{p}$,
 and $C$ depends only on $d$, $m$, $n$, $p$, $\mu$ and $\varrho(t)$ in (\ref{vmo}).
\end{corollary}

\begin{proof}
By Sobolev imbedding it follows that for $p>d$,
\begin{align*}
|\nabla^{m-1} u_\varepsilon (x)-\nabla^{m-1} u_\varepsilon (y)|
\leq C \left(\frac{|x-y|}{r}\right)^{\theta} r^{1-m}
\left\{ r^m \left(\fint_{B} |\nabla^m u_\varepsilon |^p\right)^{1/p}
+\left(\fint_{B} | u_\varepsilon |^2\right)^{1/2} \right\}
\end{align*}
 for any $ x, y \in B$.
This, together with (\ref{wmp}), gives (\ref{co51re1}).
Note that (\ref{co51re2}) is a simple consequence of (\ref{co51re1}).
\end{proof}


\section{ Asymptotic expansions of  fundamental solutions}

Let $u_\varepsilon \in H^m(B(x_0, R))$ be a weak solution of
$\mathcal{L}_\varepsilon (u_\varepsilon )=0$ in $B(x_0, R)$.
Assume that $A(y)$ satisfies  (\ref{cod1})--(\ref{cod3}) and (\ref{vmo}).
It follows from Theorem \ref{twmp} that for any $2<p<\infty$,
\begin{equation}\label{DC-1}
\left(\fint_{B(x_0, R/2)} |\nabla^m u_\varepsilon |^p\right)^{1/p}
\le C_p \left(\fint_{B(x_0, R)} |\nabla^m u_\varepsilon |^2 \right)^{1/2},
\end{equation}
where $C_p$ depends only on $d, n, m, \mu, p$ and $\varrho(t)$ in (\ref{vmo}).
By H\"older's inequality, this gives
\begin{equation}\label{DC-2}
\left(\fint_{B(x_0, r)} |\nabla^m u_\varepsilon |^2\right)^{1/2}
\le C_\sigma \left(\frac{r}{R} \right)^{-\sigma}
 \left(\fint_{B(x_0, R)} |\nabla^m u_\varepsilon |^2 \right)^{1/2}
\end{equation}
for any $0<r<R$ and for any $\sigma \in (0,1)$.
Since $A^*$ satisfies the same conditions as $A$,
estimate (\ref{DC-2}) also holds for solutions of $\mathcal{L}_\varepsilon^* (u_\varepsilon)=0$ in $B(x_0, R)$.
As a consequence, the matrix of fundamental solutions $\Gamma^{\varepsilon, A} (x,y)$ for
$\mathcal{L}_\varepsilon$ in $\mathbb{R}^d$, with pole at $y$,
exists and satisfies the estimates
 \begin{align}
&| \Gamma^{\varepsilon,A}(x,y)|\leq C|x-y|^{2m-d}, \label{l62re1}\\
&  |\nabla_x^k\Gamma^{\varepsilon,A}(x,y)| + |\nabla_y^k \Gamma^{\varepsilon,A}(x,y)| \leq C |x-y|^{2m-k-d},\label{l62re2}
\end{align}
for any $x,y\in \mathbb{R}^d, x\neq y$ and for any   $1 \leq k\leq m- 1$,
where  $C$ depends only on $d$, $n$, $m$, $\mu$,  and $\varrho (t)$
(see \cite{Auscher, ab}).
If $A(y)$ satisfies (\ref{cod1})--(\ref{cod3}) and (\ref{hol}), then for any $x,y\in \mathbb{R}^d, x\neq y,$
\begin{align}
&  |\nabla_x^m\Gamma^{\varepsilon,A}(x,y)| + |\nabla_y^m \Gamma^{\varepsilon,A}(x,y)| \leq C |x-y|^{m-d}, \label{l62re3} \\
& |\nabla_x^k \nabla_y^\ell\Gamma^{\varepsilon,A}(x,y)|\leq C |x-y|^{2m-\ell-k-d}\quad
 \text{ for any } 1 \leq \ell, k\leq m,\label{l62re4}
\end{align}
where $C$ depends only on $d, n,m, \mu, \Lambda_0$ and $\tau_0$.
This follows  readily from (\ref{l62re1}) by Theorem \ref{tlip}, as in the case of second-order elliptic systems \cite{al91,klsc}.
In (\ref{l62re1})-(\ref{l62re4}) and hereafter we assume that $2\le 2m<d$.

 In the rest of this section
 we investigate the asymptotic behavior  of $\Gamma^{\varepsilon, A}(x,y)$ and give the proof of Theorem \ref{texpan}.

\begin{lemma}\label{lemm6.3}
 Assume that $A(y)$ satisfies  (\ref{cod1})--(\ref{cod3}) and (\ref{vmo}).
 Let $u_\varepsilon \in H^m(2B; \mathbb{R}^n)$ and $u_0\in C^{2m}(2B; \mathbb{R}^n) $ such that $ \mathcal{L}_\varepsilon(u_\varepsilon)= \mathcal{L}_0(u_0)$ in $2B$ for $B=B(x_0,1)$. Then for any  $0\leq \ell\leq m-1,$
 \begin{align}\label{l63re1}
 &\|\nabla^\ell u_\varepsilon-\nabla^\ell u_0\|_{L^\infty(B)}\nonumber\\
 &\leq C \left\{\fint_{2B} |u_\varepsilon-u_0|^2 \right\}^{1/2}+ C \sum_{1\leq k\leq m} \varepsilon^{k} \| \nabla^{m+k}u_0\|_{L^\infty(2B)}+ C \sum_{0\leq k\leq l} \varepsilon^{m-k} \| \nabla^{m+\ell-k}u_0\|_{L^\infty(2B)} ,
 \end{align}
 where $C$ depends only on $d,n,m,\mu$ and $\varrho(t)$.
\end{lemma}

\begin{proof}
By translation we only need to consider the case  $x_0=0 .$ Set
\begin{align}\label{pl6301}
w_\varepsilon=u_\varepsilon-u_0- \varepsilon^m\sum_{|\gamma|=m}\chi^\gamma(\frac{x}{\varepsilon})  D^\gamma u_0.
\end{align}
Using Lemma \ref{l2.1}, we  deduce by direct computations that
\begin{align}\label{pl6302}
\mathcal{L}_\varepsilon(w_\varepsilon)&
=(-1)^{m+1}\sum_{|\alpha|=m}  D^\alpha \left\{\sum_{|\beta|=|\gamma|=m}
 \sum_{ \zeta+\eta=\beta,   |\zeta|\leq m-1}
 C(\zeta,\eta)\varepsilon^{m-|\zeta|} A^{\alpha \beta}(\frac{x}{\varepsilon}) (D^{\zeta}\chi^\gamma)(\frac{x}{\varepsilon})  D^{\eta+\gamma} u_0  \right\}  \nonumber\\
&+(-1)^m\sum_{|\alpha|=m}  D^\alpha \left\{ \sum_{|\beta|=|\gamma|=m} \sum_{  \zeta'+\eta' =\gamma, |\zeta'|\leq m-1}   C(\zeta', \eta' ) \varepsilon^{m-|\zeta'| } (D^{\zeta'}\mathfrak{B}^{\gamma \alpha\beta}) (\frac{x}{\varepsilon})D^{\eta'+\beta} u_0 \right\}.
\end{align}
In view of (\ref{co51re2}) in Corollary \ref{cor5.1},
we know that under the conditions of Lemma
\ref{lemm6.3}, $|\nabla^k \chi^\gamma|$ and $ |\nabla^k\mathfrak{B}^{\gamma\alpha \beta}|$
are bounded for  $0\leq k\leq m-1$.
  We thus derive from (\ref{pl6302}) and (\ref{co51re2}) that for  $0\leq \ell\leq m-1,$
\begin{align}\label{pl6303}
&\|\nabla^\ell w_\varepsilon\|_{L^\infty(B)}\leq C    \left\{  \left(\fint_{2B}|  w_\varepsilon|^2\right)^{1/2} +   \sum_{1\leq k\leq m} \varepsilon^k \| \nabla^{m+k}u_0\|_{L^\infty(2B)} \right\}.
\end{align}
Taking (\ref{pl6301}) into consideration, (\ref{l63re1}) follows easily from (\ref{pl6303}).
\end{proof}

\begin{lemma}\label{lemm6.4}
 Assume that $A(y)$ satisfies  (\ref{cod1})--(\ref{cod3}) and (\ref{hol}).
 Let $u_\varepsilon \in H^m(4B; \mathbb{R}^n)$, $u_0\in C^{2m,\theta} (4B, \mathbb{R}^n)$
 such that $ \mathcal{L}_\varepsilon u_\varepsilon=\mathcal{L}_0 u_0$ in $4B$ for $B=B(x_0,1)$ and $ x_0 \in\mathbb{ R}^d$.
 Then for $0<\varepsilon<1$ and any multi-index $\alpha'$ with $|\alpha'|=m$, we have
\begin{align}\label{l64re1}
\|D^{\alpha'} u_\varepsilon-D^{\alpha'} u_0-&\sum_{|\gamma|=m} (D^{\alpha'} \chi^\gamma)(x/\varepsilon)D^\gamma u_0\|_{L^\infty(B)}\nonumber\\
 &\leq  C   \left\{\fint_{4B}  |u_\varepsilon-u_0|^2  \right\}^{1/2} + C    \varepsilon ( \ln\varepsilon^{-1} +1)  \| u_0\|_{C^{2m,\theta}(4B)} .
\end{align}
\end{lemma}

\begin{proof}
 Let $ \phi\in C_c^\infty(3B)$ with $\phi=1$ in $2B$ and $|\nabla^k \phi|\leq C $  for $1\leq  k\leq 2m$,
  and let  $ w_\varepsilon $ be defined as (\ref{pl6301}). We only need to verify that $ \|\nabla^mw_\varepsilon\|_{L^\infty(B)}  $ is bounded by the RHS
of (\ref{l64re1}).
Through direct computations, we have
\begin{align*}
\mathcal{L}_\varepsilon (w_\varepsilon \phi)& =  (\mathcal{L}_\varepsilon w_\varepsilon)  \phi
+(-1)^m \sum_{|\alpha|=|\beta|=m} \sum_{\zeta+\eta=\alpha,|\eta|\geq 1} C(\zeta,\eta)D^\zeta \left[ A^{\alpha\beta} (\frac{x}{\varepsilon})  D^\beta w_\varepsilon \right] D^\eta\phi \nonumber\\
&+(-1)^m\sum_{|\alpha|=|\beta|=m} \sum_{  \zeta+\eta=\beta,|\eta|\geq 1}C(\zeta,\eta) D^\alpha \left[ A^{\alpha\beta} (\frac{x}{\varepsilon})  D^{\zeta}  w_\varepsilon  D^{\eta} \phi   \right], \nonumber
\end{align*}
which, together with (\ref{pl6302}), implies that for any $x\in B,$
\begin{align}\label{pl6401}
w_{\varepsilon i}(x)&=- \sum_{|\alpha| =m} \int_{3B} D_y^\alpha
 \left[ \Gamma_{ij}^{\varepsilon,A}(x,y) \phi(y)\right]  \Upsilon_j^\alpha(y)  dy
 - \sum_{|\alpha| =m} \int_{3B}D_y^\alpha \left[ \Gamma_{ij}^{\varepsilon,A}(x,y) \phi(y)\right]   \widetilde{\Upsilon}_j^\alpha(y)dy  \nonumber\\
&+  \sum_{|\alpha|=|\beta|=m} \sum_{\zeta+\eta=\alpha,|\eta|\geq 1} (-1)^{m+|\zeta|} C(\zeta,\eta)\int_{3B}D_y^\zeta \Gamma_{ij}^{\varepsilon,A}(x,y)   A_{jk}^{\alpha\beta} (\frac{y}{\varepsilon})  D_y^\beta w_{\varepsilon k} (y)   D_y^\eta\phi (y) dy\nonumber\\
&+  \sum_{|\alpha|=|\beta|=m}  \sum_{  \zeta+\eta=\beta,|\eta|\geq 1}C(\zeta,\eta) \int_{3B} D_y^\alpha \Gamma_{ij}^{\varepsilon,A}(x,y)   A_{jk}^{\alpha\beta} (\frac{y}{\varepsilon})  D_y^{\zeta}  w_{\varepsilon k}(y)  D_y^{\eta} \phi (y)dy \nonumber\\
&\doteq  \mathcal{I}_1(x)+ \mathcal{I}_2(x)+ \mathcal{I}_3(x)+ \mathcal{I}_4(x),
\end{align}where $ \Upsilon (y)= (\Upsilon_j^\alpha(y)), \widetilde{\Upsilon}(y) = (\widetilde{\Upsilon}_j^\alpha(y))$ with
\begin{align*}
&\Upsilon_j^\alpha(y)=\sum_{ |\beta|=|\gamma|=m}\sum_{\zeta+\eta=\beta,  |\eta|\geq 1} C(\zeta,\eta)\varepsilon^{m-|\zeta|} A_{jk}^{\alpha \beta}(\frac{y}{\varepsilon}) (D_y^{\zeta}\chi_{kl}^\gamma)(\frac{y}{\varepsilon})  D_y^{\eta+\gamma} u_{0l}(y), \\
& \widetilde{\Upsilon}_j^\alpha(y) = \sum_{ |\beta|=|\gamma|=m}\sum_{\zeta+\eta=\gamma,  |\eta|\geq 1} C(\zeta,\eta)\varepsilon^{m-|\zeta|}  (D_y^{\zeta}\mathfrak{B}_{jk}^{\gamma\alpha \beta}) (\frac{y}{\varepsilon})D_y^{\eta+\beta} u_{0k}(y).
\end{align*}
Note that
\begin{align*}
   \int_{3B} D_y^\alpha \left[ \Gamma_{ij}^{\varepsilon,A}(x,y) \phi(y)\right]    \Upsilon_j^\alpha(x)  dy\equiv0.
\end{align*}
Hence for  any multi-index $\alpha'$ with $|\alpha'|=m$,  we have
\begin{align}\label{pl6402}
D_x^{\alpha'}\mathcal{I}_1(x)&=  \sum_{|\alpha| =m} \int_{3B}D_x^{\alpha'} D_y^\alpha \left[ \Gamma_{ij}^{\varepsilon,A}(x,y) \phi(y)\right]  \left[\Upsilon_j^\alpha(y) - \Upsilon_j^\alpha(x)\right] dy\nonumber\\
& = \sum_{|\alpha|=m } \int_{3B}D_x^{\alpha'} D_y^\alpha   \Gamma_{ij}^{\varepsilon,A}(x,y) \phi(y)   \left[\Upsilon_j^\alpha(y) - \Upsilon_j^\alpha(x)\right] dy\nonumber\\
&\ \ + \sum_{|\alpha|=m } \sum_{\zeta+\eta=\alpha, |\eta|\geq 1 } C (\zeta, \eta)\int_{3B}D_x^{\alpha'} D_y^\zeta   \Gamma_{ij}^{\varepsilon,A}(x,y) D_y^\eta\phi(y)   \left[\Upsilon_j^\alpha(y) - \Upsilon_j^\alpha(x)\right] dy \nonumber\\
&\doteq \mathcal{I}_{11}+\mathcal{I}_{12}.
\end{align}
In view of (\ref{l62re4}), we have
\begin{align}\label{pl6403}
|\mathcal{I}_{11}|&\leq C \int_{B(x,\varepsilon)} \frac{|\Upsilon (x)- \Upsilon (y)|}{|x-y|^d} dy+ C\int_{3B\setminus B(x,\varepsilon)} \frac{|\Upsilon (x)- \Upsilon (y)|}{|x-y|^d} dy \nonumber\\
&\leq  C \int_{B(x,\varepsilon)} \frac{|\Upsilon (x)- \Upsilon (y)|}{|x-y|^d}  dy+ C \ln(\varepsilon^{-1}) \|\Upsilon \|_{L^\infty(3B)}\nonumber\\
&\leq C    \varepsilon ( \ln\varepsilon^{-1} +1)  \| u_0\|_{C^{2m,\theta}(4B)},
\end{align}
where for the last inequality, we have used the fact that
\begin{align*}
\|\Upsilon \|_{C^{0, \theta}(4B)}&\leq C \sum_{0\leq k\leq m-1}\varepsilon^{m-k}  \| \nabla^{2m-k} u_0\|_{C^{0,\theta}(4B)}
 +C \sum_{0\leq k\leq m-1}\varepsilon^{m-k-\theta} \|\nabla^{2m-k} u_0\|_{L^\infty(4B)},
\end{align*}
and
\begin{align}\label{pl6404}
&\|\Upsilon\|_{L^\infty(3B)}\leq C \sum_{0\leq k\leq m-1} \varepsilon^{m-k}     \| \nabla^{2m-k} u_0\|_{L^\infty(4B)} .
\end{align}
By (\ref{l62re4}) and  (\ref{pl6404}), it is easy to derive that
\begin{align*}
|\mathcal{I}_{12}|\leq C \varepsilon \| u_0\|_{C^{2m,\theta}(4B)}.
\end{align*}
This, combined with (\ref{pl6402}) and  (\ref{pl6403}), implies that
\begin{align}\label{pl6405}
|\nabla^m\mathcal{I}_1 (x)|\leq C    \varepsilon ( \ln\varepsilon^{-1} +1)  \| u_0\|_{C^{2m,\theta}(4B)} ~~ \text{ for any } x\in B.
\end{align}
In a similar way, we can show that
\begin{align}\label{pl6406}
|\nabla^m\mathcal{I}_2(x)|\leq  C    \varepsilon ( \ln\varepsilon^{-1} +1)  \| u_0\|_{C^{2m,\theta}(4B)} \ \ \text{ for any } x\in B.
\end{align}

Finally, we turn to the estimates of $\nabla^m\mathcal{I}_3(x)+\nabla^m\mathcal{I}_4(x).$
Using (\ref{l62re3}) and (\ref{l62re4}), we have
\begin{align*}
|\nabla^m\mathcal{I}_3(x)+\nabla^m\mathcal{I}_4(x)|\leq  C\sum_{0\leq k\leq m}  \int_{3B}  |\nabla^k w_\varepsilon|
 \leq C\sum_{0\leq k\leq m} \left\{ \int_{3B}  |\nabla^k w_\varepsilon|^2 \right\}^{1/2}.
\end{align*}
From (\ref{pl6302}) and Caccioppoli's inequality (\ref{l2.6re2}), we then deduce that
\begin{align*}
|\nabla^m\mathcal{I}_3(x)+\nabla^m\mathcal{I}_4(x)|\leq & C   \left\{\fint_{4B}  |u_\varepsilon-u_0|^2  \right\}^{1/2}
+C \varepsilon^{m}\|\nabla^mu_0\|_{L^\infty(4B)}\\
&+ C \sum_{0\leq k\leq m-1} \varepsilon^{m-k}\|\nabla^{2m-k}u_0\|_{L^\infty(4B)},
\end{align*}
which, combined with (\ref{pl6401}), (\ref{pl6405}) and (\ref{pl6406}), implies (\ref{l64re1}).
\end{proof}

Now we are ready to prove Theorem \ref{texpan}.

\begin{proof}[\bf Proof of Theorem \ref{texpan}]
Fix $x_0,y_0 \in \mathbb{R}^d$ and set $R=\frac{1}{8}|x_0-y_0|.$
We only need to consider the case for $ 0<\varepsilon < R$.
For otherwise, the desired estimates follow directly from
(\ref{l62re1})-(\ref{l62re4}).
Moreover, by rescaling we observe that
\begin{align}\label{ptexpan01}
 \Gamma^{\varepsilon, A}(x,y) =r^{2m-d} \Gamma^{\frac{\varepsilon}{r},A}(r^{-1}x,r^{-1}y).
\end{align}
Therefore, we  may assume that $R=1$.

Now for $F\in  C^\infty_c(\mathbb{R}^d; \mathbb{R}^n)$ with support in $B(y_0, 1),$ set
$$
 u_\varepsilon (x)=\int_{\mathbb{R}^d}  \Gamma^{\varepsilon,A}(x,y) F(y) dy, ~~ u_0 (x)=\int_{\mathbb{R}^d}   \Gamma^{0,A}(x,y) F(y) dy,$$
and define $w_\varepsilon$ as (\ref{pl6301}).
Since $ \mathcal{L}_0 u_0=0$ in $B(x_0,4)$, we see that for any  $ k\geq0$,
\begin{align} \label{ptexpan02}
\|\nabla^{m+k}u_0\|_{L^\infty(B(x_0,2))} \leq C \|\nabla^mu_0\|_{L^2(B(x_0,4))}\leq  C \|\nabla^mu_0\|_{L^2(\mathbb{R}^d)}.
\end{align}
On the other hand, since $ \mathcal{L}_\varepsilon u_\varepsilon=\mathcal{L}_0 u_0=0$ in $B(x_0,4)$, we deduce from Lemma \ref{lemm6.3} that,
\begin{align}\label{ptexpan03}
&\| u_\varepsilon-  u_0\|_{L^\infty(B(x_0,1))}\nonumber\\
&\leq  C \left\{\int_{B(x_0,2)} |u_\varepsilon-u_0|^2 \right\}^{1/2}+ C \sum_{1\leq k\leq m} \varepsilon^k \| \nabla^{m+k}u_0\|_{L^\infty(B(x_0,2))}+ C\varepsilon^m \| \nabla^mu_0\|_{L^\infty(B(x_0,2))}\nonumber\\
&\leq C \left\{\int_{B(x_0,2)} |w_\varepsilon|^2 \right\}^{1/2}+
 C \varepsilon^m\| \nabla^m u_0\|_{L^2(B(x_0,2))}+
 C \sum_{1\leq k\leq m} \varepsilon^k  \|\nabla^{m} u_0\|_{L^2(\mathbb{R}^d)} \nonumber \\
&\leq C  \|\nabla^mw_\varepsilon\|_{L^2(\mathbb{R}^d)} +C \varepsilon  \|\nabla^mu_0\|_{L^2(\mathbb{R}^d)},
\end{align}
where we have used H\"{o}lder's inequality and Sobolev imbedding for the third inequality.
Thanks to (\ref{pl6302}), we have
$$
 \|\nabla^mw_\varepsilon\|_{L^2(\mathbb{R}^d)} \leq C \sum_{1\leq k\leq m} \varepsilon^k  \|\nabla^{m+k}u_0\|_{L^2(\mathbb{R}^d)},
 $$
  which, together with
 (\ref{ptexpan03}) implies that
 \begin{align}\label{ptexpan04}
 \| u_\varepsilon-  u_0\|_{L^\infty(B(x_0,1))}
 \leq C \sum_{1\leq k\leq m} \varepsilon^k  \|\nabla^{m+k}u_0\|_{L^2(\mathbb{R}^d)}+
 C  \varepsilon  \|\nabla^mu_0\|_{L^2(\mathbb{R}^d)}.
\end{align}
 It follows by the classical Calder\'{o}n-Zygmund estimates and the fractional integral estimates (see e.g. \cite{st} Chapters II, V)
 that
\begin{align*}
&\|\nabla^{2m} u_0\|_{L^p(\mathbb{R}^d)}\leq C_p\| F\|_{L^p(\mathbb{R}^d)} \quad \text{ for } 1<p<\infty, \\
&\|\nabla^su_0\|_{L^q(\mathbb{R}^d)}\leq C_p\|F\|_{L^p(\mathbb{R}^d)} \quad
\text{ for } 1<p<\frac{d}{2m-s} \text{ and }
\frac{1}{q}=\frac{1}{p}-\frac{2m-s}{d}.
\end{align*}
Therefore, we deduce from  (\ref{ptexpan04}) that
\begin{align*}
 |  u_\varepsilon(x_0)- u_0(x_0)|\leq  \| u_\varepsilon- u_0\|_{L^\infty(B(x_0,1))} \leq C\varepsilon \|F\|_{L^2(B(y_0, 1))}.
\end{align*}
Standard duality arguments then lead to
\begin{align*}
 \| \Gamma^{\varepsilon,A}(x_0,y)- \Gamma^{0,A}(x_0,y)\|_{L^2(B(y_0,1))}\leq C\varepsilon.
\end{align*}
Since
$$
\mathcal{L}^*_\varepsilon  \Gamma^{\varepsilon,A}(x_0,y) = \mathcal{L}^*_0   \Gamma^{0,A}(x_0,y) =0 ~\text{ in } B(y_0,4),
$$
in view of Lemma \ref{lemm6.3}, we obtain that, for  $0\leq |\zeta|\leq m-1$,
\begin{align*}
&\|D_y^\zeta\Gamma^{\varepsilon,A}(x_0,y)-D_y^\zeta\Gamma^{0,A}(x_0,y)\|_{L^\infty(B(y_0,\frac{1}{2}))}\nonumber\\ &\leq C   \|\Gamma^{\varepsilon,A}(x_0,y)-\Gamma^{0,A}(x_0,y)\|_{L^2(B(y_0,1))}+ C \sum_{1\leq k\leq m} \varepsilon^{k} \| \nabla_y^{m+k} \Gamma^{0,A}(x_0,y)\|_{L^\infty(B(y_0,1))} \nonumber\\
&\ \  +  C \sum_{1\leq k\leq |\zeta|} \varepsilon^{m-k} \| \nabla_y^{m+|\zeta|-k}\Gamma^{0,A}(x_0,y)\|_{L^\infty(B(y_0,1))}  \nonumber\\
 & \leq C \varepsilon,
\end{align*}
which implies (\ref{expan1}) through simple rescaling (see (\ref{ptexpan01})).

Let us now prove (\ref{expan2}).
Note that $$ \mathcal{L}_\varepsilon  \Gamma^{\varepsilon,A}(x,y_0)= \mathcal{L}_0  \Gamma^{0,A}(x,y_0)=0  \text{ in }B(x_0,4),$$
and $\Gamma^{0,A}(x,y_0) $ is smooth with $$ \|\Gamma^{0,A}(x,y_0) \|_{C^{2m,\theta}(B(x_0,4))} \leq C .$$
We thus deduce from Lemma \ref{lemm6.4} and  (\ref{expan1}) that, for any multi-index $\xi$ with $|\xi|=m,$
\begin{align*}
&\|D^\xi_x\Gamma^{\varepsilon,A}(x,y_0)-D^\xi_x\Gamma^{0,A}(x,y_0)- \sum_{|\gamma|=m}(D_x^\xi\chi^\gamma) (x/\varepsilon)D_x^\gamma \Gamma^{0,A}(x,y_0) \|_{L^\infty(B(x_0,1))}
 \leq  C    \varepsilon ( \ln\varepsilon^{-1} +1),
\end{align*}
which implies especially that
\begin{align}\label{ptexpan05}
 |D^\xi_x\Gamma^{\varepsilon,A}(x_0,y_0)-D^\xi_x\Gamma^{0,A}(x_0,y_0)-\sum_{|\gamma|=m} (D_x^\xi\chi^\gamma) (x_0/\varepsilon)D_x^\gamma \Gamma^{0,A}(x_0,y_0) |
 \leq  C    \varepsilon ( \ln\varepsilon^{-1} +1),
\end{align} where $C$ depends only on $d, n, m, \mu, \Lambda_0,$ and $\tau_0$.
For the general case (\ref{expan2}),  we can deduce form (\ref{ptexpan01}) and (\ref{ptexpan05}) immediately by rescaling.

Finally, let us prove (\ref{expan3}) and (\ref{expan4}).  Using
\begin{equation}\label{adjoint-relation}
\Gamma^{\varepsilon, A}_{jk}(x, y) =\Gamma_{kj}^{\varepsilon, A^*} (y, x),
\end{equation}
 we may  deduce from (\ref{expan2}) that for any multi-index $\xi$ with $|\xi|=m$,
\begin{align}\label{ptexpan06}
 |D^\xi_y\Gamma^{\varepsilon,A}(x_0,y_0)-D^\xi_y\Gamma^{0,A}(x_0,y_0)- \sum_{|\gamma|=m}(D_y^\xi\chi^{*\gamma}) (y_0/\varepsilon) D_y^\gamma \Gamma^{0,A}(x_0,y_0)  | \leq C \varepsilon  \ln (\varepsilon^{-1} +1).
   \end{align}
  Writing in a  more precise way, (\ref{expan2}) and (\ref{ptexpan06}) (after rescaling) read as
 \begin{align}\label{ptexpan60}
 \Big|D^\xi_x\Gamma_{ij}^{\varepsilon,A}(x,y)
 - \sum_{|\gamma|=m} D_x^\xi\left\{ \frac{1}{\gamma!}\delta_{ik} x^\gamma
 -  \varepsilon^m\chi_{ik}^\gamma  (x/\varepsilon) \right\} D^\gamma_x\Gamma_{kj}^{0,A}(x,y)\Big|\leq C \varepsilon \frac{\ln (\varepsilon^{-1}|x-y|+1)} { |x-y|^{d+1-m} } ,\nonumber\\
  \Big| D^\xi_y\Gamma_{ij}^{\varepsilon,A}(x,y)-
  \sum_{|\gamma|=m}D^\gamma_y\Gamma_{ik}^{0,A}(x,y)D_y^\xi \left\{\frac{1}{\gamma!}\delta_{kj}y^\gamma -  \varepsilon^m\chi_{jk}^{*\gamma} (y/\varepsilon)  \right\} \Big|\leq C \varepsilon \frac{\ln (\varepsilon^{-1}|x-y|+1)} { |x-y|^{d+1-m} },
   \end{align}
   for $x, y\in \mathbb{R}^d, x\neq y.$
For $|\xi|=|\gamma|=m$, $x\in B(x_0,4),$  we set
\begin{align*}
u_{\varepsilon i}(x)=D_y^\xi \Gamma_{ij}^{\varepsilon,A}(x,y_0) \quad
\text{ and } \quad
u_0(x)= \sum_{|\gamma|=m}D^\gamma_y\Gamma_{ik}^{0,A}(x,y_0)D_y^\xi \left\{\frac{1}{\gamma!}\delta_{kj}y^\gamma
-  \varepsilon^m\chi_{jk}^{*\gamma} (y_0/\varepsilon)  \right\},
\end{align*}
where $\delta_{kj}  $ is the Kronecker function.  It follows from (\ref{ptexpan60}) that
\begin{align*}
\|u_\varepsilon-u_0\|_{L^\infty(B(x_0,1))}\leq C \varepsilon  \ln (\varepsilon^{-1} +1)  .
\end{align*}
Note that $u_0\in  C^{2m,\theta}(2B)$.  Hence Lemma \ref{lemm6.3} implies that for $0\leq |\zeta|\leq m-1$,
\begin{align} \label{ptexpan07}
&\|D_x^\zeta D_y^\xi \Gamma_{ij}^{\varepsilon,A}(x,y_0) -\sum_{|\gamma|=m}D_x^\zeta D^\gamma_y\Gamma_{ik}^{0,A}(x,y_0)D_y^\xi \left\{\frac{1}{\gamma!}\delta_{kj}y^\gamma - \varepsilon^m\chi_{jk}^{*\gamma} (\frac{y_0}{\varepsilon})  \right\} \|_{L^\infty(B(x_0,1))}\nonumber\\
&\leq C \varepsilon  \ln (\varepsilon^{-1} +1).
\end{align}
Furthermore,  by Lemma \ref{lemm6.4} we obtain that for $|\xi|=|\eta|=m,$
\begin{align}\label{ptexpan08}
&\|D_x^\eta D_y^\xi \Gamma_{ij}^{\varepsilon,A}(x,y_0) -  \Theta^{\eta,\xi}_{ij}(x,y_0)
 \|_{L^\infty(B(x_0,1))} \leq C \varepsilon  \ln (\varepsilon^{-1} +1),
\end{align}
where 
$$
\Theta^{\eta, \xi}_{ij}(x,y_0)
=\sum_{|\sigma|=m} \sum_{|\gamma|=m} D_x^\eta\left\{ \frac{\delta_{ik}}{\sigma!} x^{\sigma}-  \varepsilon^m\chi_{ik}^{\sigma}  (x/\varepsilon ) \right\}D_x^{\sigma} D^\gamma_y\Gamma_{kl}^{0,A}(x,y_0)D_y^\xi \left\{\frac{\delta_{jl}}{\gamma!}y^\gamma -  \varepsilon^m\chi_{jl}^{*\gamma} (y_0/\varepsilon)  \right\}.$$
Thanks to (\ref{ptexpan01}), we obtain (\ref{expan3}) and (\ref{expan4}) from (\ref{ptexpan07}) and (\ref{ptexpan08}) respectively by rescaling.
\end{proof}

\bibliographystyle{amsplain}

\bibliography{23.bbl}

\providecommand{\bysame}{\leavevmode\hbox to3em{\hrulefill}\thinspace}
\providecommand{\MR}{\relax\ifhmode\unskip\space\fi MR }
\providecommand{\MRhref}[2]{%
  \href{http://www.ams.org/mathscinet-getitem?mr=#1}{#2}
}
\providecommand{\href}[2]{#2}
\begin{thebibliography}{10}

\bibitem{as}
S.~N. Armstrong and Z.~Shen, \emph{Lipschitz estimates in almost-periodic
  homogenization}, Comm. Pure Appl. Math. \textbf{69} (2016), no.~10,
  1882--1923.

\bibitem{ak}
S.~N. Armstrong and C.~K. Smart, \emph{Quantitative stochastic homogenization
  of convex integral functionals}, Ann. Sci. \'Ec. Norm. Sup\'er. (4)
  \textbf{49} (2016), no.~2, 423--481.

\bibitem{Auscher}
P.~Auscher and M.~Qafsaoui, \emph{Equivalence between regularity theorems and
  heat kernel estimates for higher-order elliptic operators and systems under
  divergence form}, J. Funct. Anal. \textbf{177} (2000), 310--364.

\bibitem{al87}
M.~Avellaneda and F.~Lin, \emph{Compactness methods in the theory of
  homogenization}, Comm. Pure Appl. Math. \textbf{40} (1987), no.~6, 803--847.

\bibitem{alo}
\bysame, \emph{Homogenization of elliptic problems with {$L^p$} boundary data},
  Appl. Math. Optim. \textbf{15} (1987), no.~2, 93--107.

\bibitem{al89}
\bysame, \emph{Compactness methods in the theory of homogenization. {II}.
  {E}quations in nondivergence form}, Comm. Pure Appl. Math. \textbf{42}
  (1989), no.~2, 139--172.

\bibitem{al91}
\bysame, \emph{{$L^p$} bounds on singular integrals in homogenization}, Comm.
  Pure Appl. Math. \textbf{44} (1991), no.~8-9, 897--910.

\bibitem{ab}
A.~Barton, \emph{Gradient estimates and the fundamental solution for
  higher-order elliptic systems with rough coefficients}, Manuscripta Math.
  \textbf{151} (2016), no.~3-4, 375--418.

\bibitem{ci}
L.~A. Caffarelli and I.~Peral, \emph{On {$W^{1,p}$} estimates for elliptic
  equations in divergence form}, Comm. Pure Appl. Math. \textbf{51} (1998),
  no.~1, 1--21.

\bibitem{cs}
S.~Campanato, \emph{{Elliptic Systems in Divergence Form. Interior
  Regularity}}, Scuola Normale Superiore Pisa, Pisa, 1980.

\bibitem{dk}
H.~Dong and D.~Kim, \emph{Higher order elliptic and parabolic systems with
  variably partially {BMO} coefficients in regular and irregular domains}, J.
  Funct. Anal. \textbf{261} (2011), no.~11, 3279--3327.

\bibitem{gs2}
J.~Geng, \emph{{$W^{1,p}$} estimates for elliptic problems with {N}eumann
  boundary conditions in {L}ipschitz domains}, Adv. Math. \textbf{229} (2012),
  no.~4, 2427--2448.

\bibitem{gs}
J.~Geng, Z.~Shen, and L.~Song, \emph{Uniform {$W^{1,p}$} estimates for systems
  of linear elasticity in a periodic medium}, J. Funct. Anal. \textbf{262}
  (2012), no.~4, 1742--1758.

\bibitem{gm}
M.~Giaquinta and G.~Modica, \emph{Regularity results for some classes of higher
  order nonlinear elliptic systems}, J. Reine Angew. Math. \textbf{311/312}
  (1979), 145--169.

\bibitem{gs1}
S.~Gu, \emph{Convergence rates in homogenization of {S}tokes systems}, J.
  Differential Equations \textbf{260} (2016), no.~7, 5796--5815.

\bibitem{gsh}
S.~Gu and Z.~Shen, \emph{Homogenization of {S}tokes systems and uniform
  regularity estimates}, SIAM J. Math. Anal. \textbf{47} (2015), no.~5,
  4025--4057.

\bibitem{h}
L.~H\"ormander, \emph{{The Analysis of Linear Partial Differential Operators.
  {I}. Distribution Theory and Fourier Analysis. Reprint of the second
  edition}}, Classics in Mathematics, Springer-Verlag, Berlin, 2003.

\bibitem{jko}
V.~V. Jikov, S.~M. Kozlov, and O.~A. Oleinik, \emph{{Homogenization of
  Differential Operators and Integral Functionals}}, Springer Berlin
  Heidelberg, 1994.

\bibitem{jo}
F.~John, \emph{{Plane Waves and Spherical Means Applied to Partial Differential
  Equations}}, Interscience Publishers, New York-London, 1955.

\bibitem{klsa}
C.~Kenig, F.~Lin, and Z.~Shen, \emph{Convergence rates in {$L^2$} for elliptic
  homogenization problems}, Arch. Ration. Mech. Anal. \textbf{203} (2012),
  no.~3, 1009--1036.

\bibitem{klsj}
\bysame, \emph{Estimates of eigenvalues and eigenfunctions in periodic
  homogenization}, J. Eur. Math. Soc. \textbf{15} (2013), no.~5, 1901--1925.

\bibitem{klsa1}
\bysame, \emph{Homogenization of elliptic systems with {N}eumann boundary
  conditions}, J. Amer. Math. Soc. \textbf{26} (2013), no.~4, 901--937.

\bibitem{klsc}
\bysame, \emph{Periodic homogenization of {G}reen and {N}eumann functions},
  Comm. Pure Appl. Math. \textbf{67} (2014), no.~8, 1219--1262.

\bibitem{ks1}
C.~Kenig and Z.~Shen, \emph{Homogenization of elliptic boundary value problems
  in {L}ipschitz domains}, Math. Ann. \textbf{350} (2011), no.~4, 867--917.

\bibitem{ks2}
\bysame, \emph{Layer potential methods for elliptic homogenization problems},
  Comm. Pure Appl. Math. \textbf{64} (2011), no.~1, 1--44.

\bibitem{ks}
A.~A. Kukushkin and T.~A. Suslina, \emph{Homogenization of high-order elliptic
  operators with periodic coefficients}, Algebra i Analiz \textbf{28} (2016),
  no.~1, 89--149.

\bibitem{ps}
M.~A. Pakhnin and T.~A. Suslina, \emph{Operator error estimates for the
  homogenization of the elliptic {D}irichlet problem in a bounded domain}, St.
  Petersburg Math. J. \textbf{24} (2013), no.~6, 949--976.

\bibitem{pa}
S.~E. Pastukhova, \emph{Estimates in homogenization of higher-order elliptic
  operators}, Appl. Anal. \textbf{95} (2016), no.~7, 1449--1466.

\bibitem{pa1}
\bysame, \emph{Operator error estimates for homogenization of fourth order
  elliptic equations}, St. Petersburg Math. J. \textbf{28} (2017), no.~2,
  273--289.

\bibitem{pv}
J.~Pipher and G.~C. Verchota, \emph{Dilation invariant estimates and the
  boundary {G}\aa rding inequality for higher order elliptic operators}, Ann.
  of Math. (2) \textbf{142} (1995), no.~1, 1--38.

\bibitem{se}
E.~V. Sevost'janova, \emph{An asymptotic expansion of the solution of a second
  order elliptic equation with periodic rapidly oscillating coefficients}, Mat.
  Sb. \textbf{43} (1982), no.~2, 181--198.

\bibitem{szf}
Z.~Shen, \emph{Bounds of {R}iesz transforms on {$L^p$} spaces for second order
  elliptic operators}, Ann. Inst. Fourier \textbf{55} (2005), no.~1, 173--197.

\bibitem{szt}
\bysame, \emph{The {$L^p$} boundary value problems on {L}ipschitz domains},
  Adv. Math. \textbf{216} (2007), no.~1, 212--254.

\bibitem{sh1}
\bysame, \emph{Boundary estimates in elliptic homogenization}, Anal. PDE
  \textbf{10} (2017), no.~3, 653--694.

\bibitem{sz}
Z.~Shen and J.~Zhuge, \emph{Convergence rates in periodic homogenization of
  systems of elasticity}, Proc. Amer. Math. Soc. \textbf{145} (2017), no.~3,
  1187--1202.

\bibitem{st}
E.~Stein, \emph{Singular integrals and differentiability properties of
  functions}, Princeton Mathematical Series, No. 30, Princeton University
  Press, Princeton, 1970.

\bibitem{su}
T.~A. Suslina, \emph{Homogenization of the {D}irichlet problem for elliptic
  systems: {$L^2$}-operator error estimates}, Mathematika \textbf{59} (2013),
  no.~2, 463--476.

\bibitem{su1}
\bysame, \emph{Homogenization of the {N}eumann problem for elliptic systems
  with periodic coefficients}, SIAM J. Math. Anal. \textbf{45} (2013), no.~6,
  3453--3493.

\bibitem{Suslina2017-D}
\bysame, \emph{Homogenization of the dirichlet problem for higher-order
  elliptic equations with periodic coefficients}, Algebra i Analiz \textbf{29}
  (2017), 139--192.

\bibitem{Suslina2017-N}
\bysame, \emph{{Homogenization of the Neumann problem for higher-order elliptic
  equations with periodic coefficients}}, Preprint, arXiv1705.08295 (2017).

\bibitem{v96}
G.~C. Verchota, \emph{Potentials for the {D}irichlet problem in {L}ipschitz
  domains}, Potential theory---{ICPT} 94 ({K}outy, 1994), de Gruyter, Berlin,
  1996, pp.~167--187.

\bibitem{zk}
V.~V. Zikov, S.~M. Kozlov, O.~A. Oleinik, and Kha~T'en Ngoan, \emph{Averaging
  and {$G$}-convergence of differential operators}, Uspekhi Mat. Nauk
  \textbf{34} (1979), no.~5(209), 65--133, 256.

\end{thebibliography}

\noindent\textbf{Acknowledgments.} This paper was completed during the visits of the first and  third authors at the University of Kentucky.
They would like to extend sincere gratitude to Professor Zhongwei Shen for his guidance and warm hospitality. Special thanks also go to the Department of Mathematics for the warm hospitality and support.

\vspace{1cm}
\noindent Weisheng Niu \\
School of Mathematical Science, Anhui University,
Hefei, 230601, P. R. China\\
E-mail:weisheng.niu@gmail.com\\

\noindent Zhongwei Shen\\
Department of Mathematics, University of Kentucky,
Lexington, Kentucky 40506, USA.\\
E-mail: zshen2@uky.edu\\

\noindent Yao Xu \\
Department of Mathematics, Nanjing University,
Nanjing, 200093, P. R. China\\
E-mail:dg1421012@smail.nju.edu.cn\\

 \end{document}